\documentclass[11pt,reqno]{amsart}

\usepackage{amsmath,amsfonts,amssymb,mathrsfs, amscd,amsthm,amsbsy,amsxtra,bbm,bm, epsf,calc, xcolor}

\usepackage{latexsym}
\usepackage{fixmath}
\usepackage[style=alphabetic, natbib=true, maxnames=2]{biblatex}
\usepackage[normalem]{ulem}

\usepackage[toc,page]{appendix}
\usepackage{fullpage}
\usepackage{graphicx}
\usepackage[colorlinks=true,linkcolor=purple,
    filecolor=magenta,urlcolor=cyan,citecolor=teal,pdftitle={The Stability of the Shadow for Wasserstein Projections and its Application to Transfer Learning},pdfpagemode=FullScreen]{hyperref}
\usepackage[dvipsnames]{xcolor}

\usepackage{enumerate}
\usepackage{todonotes}

\usepackage[T1]{fontenc}
\usepackage{fourier}
\usepackage{bbm}

\usepackage{epsfig}
\usepackage{dsfont}
\usepackage{tikz}
\usepackage{xr}
\usepackage{caption}
\usepackage{subcaption}
\usepackage{cleveref}
\usepackage{epstopdf}

\usepackage{changes}

\usepackage{algorithm}
\usepackage{algpseudocode}

\usepackage{tikz-cd} 
\usepackage{soul}
\usepackage{pst-node}
\providecommand{\keywords}[1]{\textbf{\emph{Keywords---}} #1}

\numberwithin{equation}{section}

\renewcommand{\epsilon}{\eps}

\newtheorem{definition}{Definition}[section]

\newtheorem{remark}[definition]{Remark}
\newtheorem{theorem}[definition]{Theorem}
\newtheorem{assumption}[definition]{Assumption}
\newtheorem{lemma}[definition]{Lemma}
\newtheorem{proposition}[definition]{Proposition}
\newtheorem{corollary}[definition]{Corollary}

\newcommand\eps{\varepsilon}

\def\?#1{}
\makeatletter
\def\whp{w.h.p\@ifnextchar-{.}{\@ifnextchar.{.\?}{\@ifnextchar,{.}{\@ifnextchar){.}{\@ifnextchar:{.:\?}{.\ }}}}}}
\makeatother

\makeatletter
\def\Whp{W.h.p\@ifnextchar-{.}{\@ifnextchar.{.\?}{\@ifnextchar,{.}{\@ifnextchar){.}{\@ifnextchar:{.:\?}{.\ }}}}}}
\makeatother

\DeclareMathOperator*{\argmin}{argmin}

\def\E{{\mathcal E}}

\begin{document}

\title{The Stability of Wasserstein Shadow Projection and its Application to Transfer Learning}


\author{Jakwang Kim}
\address{Jakwang Kim, {\tt jakwangkim@cuhk.edu.cn}, School of Data Science, The Chinese University of Hong Kong, Shenzhen.}
\date{\today}

\begin{abstract}
In this paper, we study the stability of the shadow, a projection of a probability measure onto the set of couplings with prescribed marginals in the Wasserstein space, and its application to transfer learning. The shadow was introduced by \citet{EcksteinNutz2022} to analyze the stability of the Sinkhorn algorithm and was recently revisited by \citet{kim2026extensioncouplingprojectionoptimal} for statistical applications, particularly for reconstructing joint distributions from unpaired marginal observations. Under suitable regularity assumptions on Brenier potentials, we establish quantitative continuity of the shadow. As a key ingredient in our analysis, we also establish a stability bound for optimal transport maps under a smoothness assumption based on the previous results \citet{Li_Nochetto_2021, balakrishnan2025stabilityboundssmoothoptimal}. Also, we show the uniqueness of optimal transport projection provided that source measure is absolutely continuous, and the separable cost satisfies some regular conditions introduced by \cite{gangbo1996geometry}, which includes the $p$-th power of distance for $1 < p <\infty$.

As an application, we investigate transfer learning from the perspective of shadow projection. Specifically, we consider a transport-based domain adaptation model with unpaired target covariate and response observations. We bound the generalization error of it by the approximation error and the distance between the shadow and its empirical counterpart, and also show and the consistency of the proposed estimator under smooth the Brenier potentials transporting the source marginals to their target counterparts. Additionally, assuming their strong convexity, we derive sample-complexity bounds and concentration of the generalization error. In particular, our analysis applies to suitable families of strongly log-concave measures.
\end{abstract}

\subjclass[2020]{62G05, 49Q22, 62D99}
\keywords{Optimal transport projection, Wasserstein distance projection, Shadow, Stability of Shadow, Sample complexity of Shadow, Transfer learning, Domain adaptation}
\maketitle

\section{Introduction}
In this paper, we study the stability of the \emph{shadow}, a construction of solutions to a \emph{Wasserstein distance projection} problem, which was introduced by \citet{EcksteinNutz2022} to analyze the stability of the Sinkhorn algorithm (and further extended in \cite{bayraktar_StabilitySampleComplexity_2025}), and recently revisited by \citet{kim2026extensioncouplingprojectionoptimal} for statistical applications for the reconstruction of joint distributions from unpaired marginal observations.

Although projection is a natural geometric operation, understanding its behavior on metric spaces of nonnegative curvature, such as Wasserstein space, is challenging. It is well known that metric projections are nonexpansive on a convex subset of nonpositively curved metric spaces \cite[Proposition 2.6]{Sturm_2003}. Since Wasserstein space has non-negative lower curvature, it is shown that projection on any geodesically convex subset of Wasserstein space need not be nonexpansive in general \cite[Theorem 4.9]{AdveMészáros2025}.

We establish quantitative continuity of the shadow in terms of input data. Precisely, under the assumptions on smoothness of the Brenier potentials, \(\frac{1}{2}\)-Hölder continuity holds. Lipschitz continuity can be achieved if additionally strong convexity of the Brenier potentials is assumed. As a byproduct of our analysis, we also derive a stability bound for optimal transport maps under a smoothness assumption based on \citet{Li_Nochetto_2021, balakrishnan2025stabilityboundssmoothoptimal}.

Also, we show the uniqueness of optimal transport projection provided that source measure is absolutely continuous, and the separable cost satisfies some regular conditions introduced by \cite{gangbo1996geometry}, which includes the $p$-th power of distance for $1 < p <\infty$.

Beyond establishing the stability of shadow projections, we study an application of shadow projection to \emph{transfer learning}. Transfer learning theory identifies structures through which source information improves target prediction. \citet{LCL2022} study proximity between high-dimensional regression coefficients via transfer learning framework, whereas \citet{MPR2016} analyze the benefits of learning representations shared across tasks. \citet{CGG2023} provide a complementary framework where an input transformation and an output transformation adapt a pretrained source predictor.

Optimal transport gives these transformations a geometric interpretation. \citet{OTDA2017} align source and target covariate distributions, and \citet{RHS2017} derive Wasserstein-based adaptation bounds that also account for source–target predictive compatibility. \citet{JDOT2017} extend alignment to joint distributions through optimal transport, which compares the source joint law with a target joint law induced by a candidate predictor. \citet{LHS2012} instead separate marginal and dependence changes through nonparametric copulas. These approaches illustrate why aligning covariate marginals alone does not identify the target regression relationship.

Recent work has sought to gain a broader understanding of incomplete pairing and quantitative transport estimation. \citet{SS2024} study permuted and unlinked regression, using cyclical monotonicity and optimal transport to identify regression functions without observed covariate–response correspondence, which is a similar setting to our model. \citet{rigollet2025sample} transfer source labels through an entropic coupling of covariates and establish prediction guarantees,  including a bound of $\frac{1}{\sqrt{n}}$ $L^2$ error under fixed regularization and boundedness assumptions. \citet{CGTW2026} derive sample-complexity guarantees for optimal transport constrained input–predictor–output compositions.

We establish sample-complexity guarantees for a noiseless transfer-learning model formulated as \emph{transport-based domain adaptation with unpaired target covariate and response observations}. The model combines paired source observations with two separately observed target marginal samples and identifies the target joint distribution as the shadow of the source joint distribution, with regression function $G=T_2\circ F\circ T_1^{-1}$. Although the input--predictor--output architecture and OT-based transfer learning have direct antecedents~\cite{CGG2023,CGTW2026,rigollet2025sample}, our contribution is a shadow-projection analysis that connects distributional estimation errors to prediction accuracy in unpaired target observations setting. Under smoothness of the Brenier potentials transporting the source marginals to their target counterparts, we bound the generalization error of the model by the approximation error and the Wasserstein discrepancy between the shadow and its empirical counterpart, hence attain the consistency of our estimator: under additional assumptions on the source estimator and the admissibility of the population potentials, we also establish consistency: the proposed estimator admits a subsequence converging almost surely to $G$ in $L^2(\mu_1)$. Adding strong convexity assumption further, we derive nonasymptotic bounds on the target generalization error, yielding sample-complexity guarantees.  These results provide quantitative guarantees for transferring source predictive dependence to a target population without observing paired target covariates and responses.

\subsection{Notation}\label{subsec: notation}
\begin{itemize}
    \item $\bm{\mathcal{X}} := \mathcal{X}_1 \times \dots \times \mathcal{X}_K \subseteq \mathbb{R}^{d_1} \times \dots \times \mathbb{R}^{d_K}$.
    \item $\bm{x}:=(x_i), \bm{y}:=(y_i) \in \bm{\mathcal{X}}$.
    \item $\mathcal{P}(\mathcal{S})$ denotes the set of Borel probability measures on $\mathcal{S}$.
    \item $\mathcal{P}_p(\bm{\mathcal{X}}):= \{ \mu \in \mathcal{P}(\bm{\mathcal{X}}) : \mathbb{E}_\mu \| X\|^p <\infty \}$, $\mathcal{P}^{ac}(\bm{\mathcal{X}}):= \{\mu \in \mathcal{P}(\bm{\mathcal{X}}): \mu \ll \mathrm{Lebesgue} \}$, and $\mathcal{P}^{ac}_p(\bm{\mathcal{X}}):= \mathcal{P}_p(\bm{\mathcal{X}}) \bigcap \mathcal{P}^{ac}(\bm{\mathcal{X}})$.
    \item $\bm{\mu}:= (\mu_i), \bm{\nu}:= (\nu_i) \in \prod_{i=1}^K \mathcal{P}(\mathcal{X}_i)$ denote the vectors of $K$-marginals.
    \item For $\rho, \xi \in \mathcal{P}(\bm{\mathcal{X}})$, $\rho_i, \xi_i \in \mathcal{P}_2(\mathcal{X}_i)$ denote the $i$-th marginals of $\rho$ and $\xi$, respectively.
    \item $\Pi(\bm{\mu}):= \Pi(\mu_1, \dots, \mu_K)$ denotes the set of couplings whose marginals are $\mu_i$'s.
    \item $\Pi(\bm{\mu}, \rho) \subseteq \mathcal{P}(\bm{\mathcal{X}} \times \bm{\mathcal{X}})$ denotes the set of couplings whose first marginal on $\bm{\mathcal{X}}$ is contained in $\Pi(\bm{\mu})$, and the second one is $\rho$.
    \item $\Pi_{opt}(\mu,\nu)$ denotes the set of optimal couplings achieving $W_2(\mu,\nu)$.
\end{itemize}

\subsection{Preliminaries: Optimal Transport}
Let us review optimal transport (OT) briefly. OT was proposed by \citet{monge1781memoire} and developed by \citet{kantorovich1942translocation, kantorovich1948problem}, who developed the Kantorovich relaxation, for which he received the Nobel Prize in economic sciences. Since then, OT has interwoven partial differential equations, geometry, and probability~\cite{Brenier91, mccann1994convexity, Caffarelli_contraction, Benamou_brener2000, Otto2000GeneralizationOA, otto2001geometry, Lott2004RicciCF, Figalli2010Invent, Kim2010JEMS, Figall2011duke, Caffarelli2013Reine, Figalli2015CMP, GOZLAN20173327}. OT has also become important in data science: optimization of mean field two-layer neural networks~\cite{mei2018mean, chizat2018global, sirignano2020mean, Rotskoff2022CPAM}, diffusion models~\cite{de2021diffusion, wang2021deep, hamdouche2023generative}, variational inference~\cite{lambert2022variational, diao2023forward}, and robust training~\cite{pydi2021the, Trillos2020AdversarialCN, jakwang_2024existence, jakwang_2024optimal}.

Formally, given two spaces $S_1$ and $S_2$, $\mu \in \mathcal{P}(S_1), \nu \in \mathcal{P}(S_2)$ and a cost function $c : S_1 \times S_2 \longrightarrow \mathbb{R}$, the optimal transport problem is to minimize the following transport cost:
\[
    \inf_{T:S_1 \to S_2, \,\, T_\#(\mu) = \nu} \int_{S_1} c(s, T(s)) d\mu(s).
\]
Despite its direct interpretation, Monge’s problem need not admit a solution. Kantorovich proposed a convex relaxation, the modern form of OT, as follows \cite{kantorovich1942translocation, kantorovich1948problem}:
\[
    \inf_{\pi \in \Pi(\mu, \nu)}\int_{S_1 \times S_2} c(s_1, s_2) d\pi(s_1, s_2).
\]
Under standard assumptions, the Kantorovich problem admits a solution: \cite{villani2003topics, villani2008optimal}.

As an infinite dimensional linear program (hence convex program), OT admits a dual formulation, that provides information about the primal problem. The dual of OT is given as
\[
    \sup_{\phi \in C(S_1), \psi \in C(S_2) : \phi(s_1) + \psi(s_2)\leq c(s_1, s_2)} \int_{S_1} \phi(s_1) d\mu(s_1) + \int_{S_2} \psi(s_2) d\nu(s_2).
\]

A special case of the Kantorovich problem with $S_1=S_2$ and $c=d^p$ where $d$ is a metric on $S$ is called \emph{$p$-Wasserstein distance} for $1 \leq p \leq \infty$, or \emph{Kantorovich-Rubinstein distance}, which is defined over $\mathcal{P}_p(S)$, the set of probability measures with finite $p$-th order moment. In this paper, we focus on $2$-Wasserstein distance equipped with the usual $\ell^2$ Euclidean metric $\| \cdot \|$:
\[
\begin{aligned}
    {W}_2(\mu, \nu)& := \inf_{\gamma \in \Pi(\mu, \nu)} \left( \int_{S \times S} \|x - y \|^2 d\gamma(x,y) \right)^{\frac{1}{2}}.
\end{aligned}
\]
It is well known that $\mathcal{P}_p(S)$ equipped with $W_p$ is a Polish space: in particular, $(\mathcal{P}_2(\mathcal{S}), {W}_2)$ exhibits a Riemannian structure \cite{otto2001geometry}.

One of the most important results in optimal transport is Brenier's theorem \cite{Brenier91}, which characterizes the optimal dual potentials and the optimal transport map for $2$-Wasserstein distance by a convex function and its gradient.

\begin{theorem}[Brenier's theorem \cite{Brenier91}]\label{thm: Brenier's thm}
Let $\mu \in \mathcal{P}^{ac}_2(\mathbb{R}^d), \nu \in \mathcal{P}_2(\mathbb{R}^d)$. For $W_2(\mu,\nu)$, there is a convex function $\varphi$ (called a Brenier potential) such that $T:= \nabla \varphi$ is a unique optimal transport map from $\mu$ to $\nu$ $\mu$-almost everywhere. Furthermore, there is a unique optimal transport plan $\pi^*$ such that $\pi^* = (\mathrm{Id}, T)_\#(\mu)$.

Assume further that $\nu \in \mathcal{P}^{ac}_2(\mathbb{R}^d)$ too. Then, there is a convex function $\psi$ from $\nu$ to $\mu$ which generates a unique optimal transport map $S$ from $\nu$ to $\mu$ in $\nu$-almost everywhere sense. In addition, it holds for $\mu$-almost every $x$ and $\nu$-almost every $y$ that
\[
    \nabla \psi \circ \nabla \varphi(x) = x, \quad \nabla \varphi \circ \nabla \psi(y) = y.
\]
In other words, $\nabla \psi = \nabla \varphi^* = (\nabla\varphi)^{-1}$ exists for $\nu$-almost every $y$. Here $\varphi^*$ denotes the convex conjugate of $\varphi$.    
\end{theorem}

An important question related to Brenier's theorem and Monge–Ampère equation is about the regularity of Brenier potential: when does it possess strong regularity? {C}affarelli contraction theorem answers such questions for the family of log-concave measures \cite{Caffarelli_contraction, Caffarelli_contraction_erratum}, which is recently revisited in \cite{entropic_proof_Caffarelli_contration, covariance_proof_Caffarelli_contration} by entropic regularization techniques.

\begin{definition}
Given a probability measure $\mu \propto \exp (-V)$, we say that $\mu$ is a log-concave measure if $V$ is convex. If $V$ is strongly convex, $\mu$ is called a strongly log-concave measure.    
\end{definition}

\begin{theorem}[{C}affarelli's contraction theorem \cite{Caffarelli_contraction, Caffarelli_contraction_erratum}]\label{thm: Caffarelli_contraction}
Let $\mu \propto \exp (-V)$ and $\nu \propto \exp (-W)$ have smooth densities on $\mathbb{R}^d$, with $0 \prec \alpha_1 \mathrm{Id} \preceq \nabla^2 V \preceq \beta_1 \mathrm{Id}$ and $0 \prec \alpha_2 \mathrm{Id} \preceq \nabla^2 W \preceq \beta_2 \mathrm{Id}$. Then a Brenier potential $\varphi$ from $\mu$ to $\nu$ is $\sqrt{\alpha_1/\beta_2}$-strongly convex and $\sqrt{\beta_1/\alpha_2}$-smooth (i.e., $\nabla \varphi$ is $\sqrt{\beta_1/\alpha_2}$-Lipschitz).
\end{theorem}

We end this section by introducing the recent progress of statistical optimal transport, which aims at understanding the behavior of the rate of convergence of Wasserstein distance of empirical measures. Since \citet{dudley69} paved this field, there have been numerous papers in this direction \cite{MR2280433, MR2861675, Dereich_Scheutzow_Schottstedt2013, MR3189084, NF_AG_rate_Wasserstein, MR4359822, merigot2021non, MR4255123, MR4441130, JW_FB_sample_rates}. It is well known that the rate of convergence of Wasserstein distance of empirical measures suffers the curse of dimensionality: if $\mu$ is absolutely continuous measure on $[0,1]^d$, and $\mu^n$ denotes the empirical distribution of i.i.d. $n$ samples of $\mu$, then $W_p(\mu^n, \mu) \sim n^{-\frac{1}{d}}$ for all sufficiently large $d.$ Here we use the expectation and concentration bounds of the expectation of the Wasserstein distance of the empirical measures and its concentration inequalities that are developed in \cite[Theorems 1 and 2]{NF_AG_rate_Wasserstein} with finite higher moments assumption.

\begin{lemma}\cite[Theorems 1 and 2]{NF_AG_rate_Wasserstein}\label{lem: rate of convergence of empirical measures}
Let $\mu$ be a probability measure on $\mathbb{R}^{d}$ with a finite $q$-th moment, denoted by $M_q(\mu)$, for some $q > 4$, and let $\mu^{n}$ be an empirical measure for $\mu$. 
\begin{enumerate}
    \item[(i)] For $n\geq1$, define
\begin{equation}\label{eq:shadow-empirical-rate}
 r_d(n):=
\begin{cases}
n^{-\frac{1}{4}},&d<4,\\
n^{-\frac{1}{4}}\sqrt{\log(1+n)},&d=4,\\
n^{-\frac{1}{d}},&d>4.
\end{cases}
\end{equation} 
There is a constant $C>0$ depending on $p,d,q$ such that for all $n \geq 1$, 
\begin{equation}\label{eq: convergence rate of FG15}
    \mathbb{E}[W_2(\mu^n,\mu)] \leq C M^{1/q}_q(\mu) r_d(n)
\end{equation}
    \item[(ii)] Fix $q>4$ and $\zeta\in(0,q-4)$, and put $\gamma:=\frac{q-\zeta}{2}>2$. For $t>0$, let
\begin{equation}\label{eq:hp-h-function}
h_d(t):=
\begin{cases}
t^2,&d<4,\\[1mm]
\displaystyle\left(\frac{t}{\log(2+\frac{1}{t})}\right)^2,&d=4,\\[2mm]
t^{\frac{d}{2}},&d>4.
\end{cases}
\end{equation}
Then, for any $n \geq 1$ and any $t \in (0,\infty)$,
\begin{equation}\label{eq:hp-empirical-tail}
    \mathbb{P}(W^2_2(\mu^n,\mu) \geq t) \leq Ce^{-cnh_d(t)}\mathbf 1_{\{t\leq1\}}
     +Cn^{1-\gamma}t^{-\gamma}.
\end{equation}
Here, the positive constants $C$ and $c$ depend only on $p$, $d$, $q$, $M_q(\mu)$ and $r$.
\end{enumerate}
\end{lemma}

\section{Stability and uniqueness of shadow}
\subsection{Stability}
Our first result concerns the stability of the shadow, a projection of probability measures onto the set of couplings in the $2$-Wasserstein distance, under the assumption of sufficiently regular Brenier potentials. More precisely, we establish the quantitative stability of the shadow in terms of the input data. The shadow projection was originally introduced by \citet{EcksteinNutz2022} to analyze the stability of the Sinkhorn algorithm, and has recently been revisited for statistical applications by \citet{kim2026extensioncouplingprojectionoptimal} to reconstruct the missing joint distribution structure.

The \emph{separable} product metric $d_{2}$ over $\bm{\mathcal{X}}=\prod_{i=1}^K \mathcal{X}_i$ is given as 
\begin{equation*}\label{eq: metric over X}
    d_{2}(\bm{x}, \bm{z}):=\left( \sum_{i=1}^K \| x_i - z_i\|^2\right)^\frac{1}{2}.
\end{equation*}
Throughout this paper, the $2$-Wasserstein distance $W_2$ is defined on $\mathcal{X}_i$ and $\bm{\mathcal{X}}$ with respect to $\ell_2$ and $d_{2}$, respectively. For notational convenience, we also use
\[
    \bm{W}_2(\bm{\mu}, \bm{\nu}):= \left( \sum_{i=1}^K W^2_2(\mu_i, \nu_i) \right)^{\frac{1}{2}}.
\]

Consider the following $2$-Wasserstein projection problem: given $\rho \in \mathcal{P}_2(\bm{\mathcal{X}})$ and $\bm{\mu} \in \prod_{i=1}^K \mathcal{P}_2(\mathcal{X}_i)$,
\begin{equation}\label{problem: Wasserstein projection}
    \min_{\pi \in \Pi(\bm{\mu})} W_2(\pi, \rho).
\end{equation}
It is straightforward to check that \eqref{problem: Wasserstein projection} is equivalent to the following multimarginal optimal transport (MOT) problem:
\begin{equation}\label{eq: MOT formulation of Wasserstein projection}
    \min_{\gamma \in \Pi(\bm{\mu}, \rho)} \left\{ \int_{\bm{\mathcal{X}} \times \bm{\mathcal{X}}}  d_{2}^2(\bm{x}, \bm{y}) d\gamma \right\}^{\frac{1}{2}}
\end{equation}
where the marginal constraints should be understood as $(\mathrm{P}_{\bm{x}})_\#(\gamma) = \pi$ for some $\pi \in \Pi(\bm{\mu})$ and $(\mathrm{P}_{\bm{y}})_\#(\gamma) = \rho$. Note that $\gamma_i= (\mathrm{P}_{x_i}, \mathrm{P}_{y_i})_\# \gamma \in \Pi(\mu_i, \rho_i)$ where $\rho_i$ is the $i$-th marginal of $\rho$. Denoting an optimal multimarginal coupling of \eqref{eq: MOT formulation of Wasserstein projection} by $\gamma^*$, an  optimal solution $\pi^*$ of \eqref{problem: Wasserstein projection} is obtained from $\gamma^*$ by $\pi^* := (\mathrm{P}_{\bm{x}})_\# (\gamma^*)$.

\begin{remark}
One can study $p$-Wasserstein projection for any $1 \leq p \leq \infty$ with the suitable separable product metric $d_p$. Here we restrict attention to $2$-Wasserstein projection.    
\end{remark}

\begin{remark}\label{rmk: non uniqueness}
The equivalent MOT problem shows that \eqref{problem: Wasserstein projection} may admit multiple solutions. Sufficient conditions for the uniqueness of MOT problem have been studied in \cite{Kim_mongeMOT2014, MOT_Brendan2015, MOT_viagraph2023, generalMOT2024}, which include the absolute continuity of one of the marginals, the twist condition and semi-concavity of the cost function. 
\end{remark}

Leveraging this separable structure, \citet{EcksteinNutz2022} characterize a particular class of solutions to \eqref{problem: Wasserstein projection}, which they call the \emph{shadow}, in order to analyze the stability of the Sinkhorn algorithm (see also its extension to general divergence regularization in \cite{bayraktar_StabilitySampleComplexity_2025}). They define the shadow of $\rho$ onto $\Pi(\bm{\mu})$ and show that it is indeed a projection for \eqref{problem: Wasserstein projection}.

Recall $\rho_i$ for each $i=1, \dots, K$ is the $\mathcal{X}_i$-the marginal of $\rho$.

\begin{definition}[Shadow \cite{EcksteinNutz2022}]
Let $\gamma_i(dx_i, dy_i) =  \kappa_i(dx_i |y_i) \rho_i(dy_i)$ be an optimal coupling for $W_2(\mu_i, \rho_i)$, and $\kappa(d\bm{x} | \bm{y}):=  \kappa_1(dx_1 |y_1) \otimes \dots \otimes \kappa_K(dx_K |y_K)$. Then, a shadow $\pi^*$ of $\rho$ onto $\Pi(\bm{\mu})$ is given as
\[
    \pi^*(d\bm{x}) = \int_{\bm{\mathcal{X}}} \kappa(d\bm{x} | \bm{y})\rho(d\bm{y}),
\] 
i.e., the first marginal of $\gamma^*(d \bm{x}, d\bm{y}) := \kappa(d\bm{x} | \bm{y}) \rho(d\bm{y})$.

Let $\mathcal{S}(\rho; \bm{\mu})$ denote the shadow of $\rho$ onto $\Pi(\bm{\mu})$ (if multiple shadows exist, one of them will be selected accordingly). 
\end{definition}

It holds that $\pi^*$ and $\gamma^*$ are optimal for \eqref{problem: Wasserstein projection} and \eqref{eq: MOT formulation of Wasserstein projection}, respectively, and the optimal values are equal: indeed, $\eqref{problem: Wasserstein projection}=\eqref{eq: MOT formulation of Wasserstein projection}= \left( \sum_{i=1}^K W_2^2(\rho_i, \mu_i) \right)^{1/2}$. See \cite[Lemma 3.2]{EcksteinNutz2022}.

\begin{remark}
As stated in \Cref{rmk: non uniqueness}, even multiple shadows may exist due to the multiplicity of optimal couplings $\gamma_i$ between $\mu_i$ and $\rho_i$. We will see in \Cref{thm:uniqueness of projection} that when each $\rho_i$ is absolutely continuous and $\rho_i$'s and $\mu_i$'s have finite moments (see a condition \eqref{eq: moment condition for h}), the shadow projection uniquely exists, and indeed it is a unique projection for \eqref{problem: Wasserstein projection}.    
\end{remark}

First, we establish the quantitative stability of the shadow with respect to the input data under the regularity condition on Brenier potentials, which is standard in statistical optimal transport \cite{Li_Nochetto_2021, plugin_estimation_AOS, MR4901218, balakrishnan2025stabilityboundssmoothoptimal}. Let $\rho, \xi \in \mathcal{P}(\bm{\mathcal{X}})$ and $\bm{\mu}, \bm{\nu} \in \prod_{i=1}^K \mathcal{P}_2(\mathcal{X}_i)$, and $\mathcal{S}(\rho; \bm{\mu})$ be the shadow projection of $\rho$ onto $\Pi(\bm{\mu})$. Let $S(\xi; \bm{\nu})$ denote \emph{any projection of $\xi$ onto $\Pi(\bm{\nu})$}. Observe that since $S(\xi; \bm{\nu}) \in \Pi(\bm{\nu})$, a trivial lower bound follows from \cite[Lemma 3.2]{EcksteinNutz2022} that
\begin{align*}
    W_2(\mathcal{S}(\rho; \bm{\mu}), S(\xi; \bm{\nu})) \geq \min_{\pi \in \Pi(\bm{\nu})} W_2(\mathcal{S}(\rho; \bm{\mu}), \pi) = \bm{W}_2(\bm{\mu}, \bm{\nu}).
\end{align*}
Note that this lower bound is sharp when $\rho$ and $\xi$ are product measures.

The first theorem gives a nontrivial upper bound on the distance between shadows. We emphasize that the below stability of shadow projection hold against not only another shadow projection but \emph{any projection}.

\begin{theorem}[Stability of shadow]\label{thm:smooth-shadow-stability}
Let $\rho,\xi\in\mathcal{P}_2(\bm{\mathcal{X}})$ and $\bm{\mu}, \bm{\nu} \in \prod_{i=1}^K\mathcal{P}_2(\mathcal{X}_i)$. Assume that for each $i=1, \dots ,K$ there is a Brenier potential $\varphi_i\in C^1$ from $\rho_i$ to $\mu_i$. Let $S(\xi; \bm{\nu})$ be any projection $\xi$ onto $\Pi(\bm{\nu})$.
\begin{enumerate}
    \item[(i)] Assume that each $\varphi_i$ is $\beta_i$-smooth. Then, it holds that
\begin{equation}\label{eq:smooth-shadow-stability}
    \bm{W}_2(\bm{\mu}, \bm{\nu})\leq W_2 ( \mathcal{S}(\rho; \bm{\mu}), S(\xi; \bm{\nu}) ) \leq  2 \max_i\beta_i W_2(\rho, \xi) + \bm{W}_2(\bm{\mu}, \bm{\nu}) +  2 \sqrt{\max_i \beta_i s(\xi) \bm{W}_2(\bm{\mu}, \bm{\nu})}
\end{equation}
where $s(\xi):=  \sqrt{\mathrm{tr}\,\mathrm{Cov}(\xi)}$.
    \item[(ii)]  Assume that each $\varphi_i$ is $\alpha_i$-strongly convex and $\beta_i$-smooth. Then, it holds that
\begin{equation}\label{eq:convex-smooth-shadow-stability}
    \bm{W}_2(\bm{\mu}, \bm{\nu})\leq W_2 ( \mathcal{S}(\rho; \bm{\mu}),S(\xi; \bm{\nu}) ) \leq 2 \max_i \beta_i W_2(\rho, \xi)  + \max_{i}\sqrt{\frac{\beta_i}{\alpha_i}} \bm{W}_2(\bm{\mu, \bm{\nu}}).
\end{equation}
\end{enumerate}
\end{theorem}

Combining \Cref{thm:smooth-shadow-stability} (ii) with \Cref{thm: Caffarelli_contraction}, we have the stability of shadow projection for the case that target measures belong to a family of strongly log-concave measures. Note that if the source measure $\rho$ satisfies $\rho \propto \exp \left( - W \right)$ such that $\nabla W$ is $\beta$-Lipschitz, then each of its marginal $\rho_i$ satisfies $\rho_i \propto \exp (- W_i)$ for some $W_i$ such that $\nabla W_i$ is also $\beta$-Lipschitz. We present its proof in \Cref{prop: some properties of log-concave} for the reader's convenience.

\begin{corollary}[Stability of Shadow projection for strongly log-concave measures]
Let $\rho \propto \exp (-W)$ and $\mu_i \propto \exp (-V_i)$ have smooth densities respectively such that $0 \prec a \mathrm{Id} \preceq \nabla^2 W \preceq b \mathrm{Id}$ and $0 \prec \alpha_i \mathrm{Id} \preceq \nabla^2 V_i \preceq \beta_i \mathrm{Id}$ for some $0 < \alpha_i \leq \beta_i < \infty$ for $i=1, \dots, K$ and $0 < a \leq b <\infty$. Let $\varphi$ be a Brenier potential from $\rho_i$ to $\mu_i$. Then, for any $\xi \in \mathcal{P}_2(\bm{\mathcal{X}})$ and $\bm{\nu} \in \prod \mathcal{P}_2(\mathcal{X}_i)$, it holds that
\begin{equation*}
    W_2(\mathcal{S}(\rho; \bm{\mu}), S(\xi; \bm{\nu})) \leq 2\max_i \sqrt{\frac{b}{\alpha_i}} W_2(\rho, \xi)  + \max_{i} \left( \frac{b\beta_i}{a \alpha_i} \right)^{\frac{1}{4}}\bm{W}_2(\bm{\mu, \bm{\nu}}).
\end{equation*}
\end{corollary}

Combined with \Cref{lem: rate of convergence of empirical measures}, the expected sample complexity and the concentration of empirical shadow are obtained.

Recall $r_d(n)$ and $h_d(t)$ defined in \eqref{eq:shadow-empirical-rate} and \eqref{eq:hp-h-function}, respectively.

\begin{corollary}[Sample complexity of the empirical shadow]
\label{cor:shadow-sample-complexity}
Assume that $\rho$ and $\bm{\mu}$ have finite moments of order $q$ for some $q > 4$. 
\begin{enumerate}
    \item[(i)] Under the same setting of \Cref{thm:smooth-shadow-stability} of (i), there are constants $C_\rho$ and $C_{\mu_i}$'s independent of $m,n_1,\ldots,n_K$, such that
\begin{equation}\label{eq:shadow-mean-explicit}
\begin{aligned}
    \mathbb E W_2(\mathcal{S}(\rho; \bm{\mu}), \mathcal{S}(\rho^m; \bm{\mu}^{\bm{n}})) \leq  2C_\rho \max_i \beta_i r_d(m) + \sum_{i=1}^K C_{\mu_i} r_{d_i}(n_i) + 2 \left(\max_i\beta_i s(\rho) \sum_{i=1}^K C_{\mu_i}  r_{d_i}(n_i)\right)^{\frac{1}{2}}.
\end{aligned}
\end{equation}
Consequently, for sufficiently large $m, n_1, \dots, n_K$, there is a constant $C > 0$ independent of the sample sizes such that
\begin{equation*}
    \mathbb E W_2(\mathcal{S}(\rho; \bm{\mu}), \mathcal{S}(\rho^m; \bm{\mu}^{\bm{n}})) \leq C\left\{  r_d(m)+\sum_{i=1}^K  r_{d_i}(n_i)^{\frac{1}{2}} \right\}.
\end{equation*}
    \item[(ii)] Under the same setting of \Cref{thm:smooth-shadow-stability} of (ii), there are constants $C_\rho$ and $C_{\mu_i}$'s independent of $m,n_1,\ldots,n_K$, such that
\begin{equation}\label{eq:convex-shadow-mean-explicit}
     \mathbb E W_2(\mathcal{S}(\rho; \bm{\mu}), \mathcal{S}(\rho^m; \bm{\mu}^{\bm{n}})) \leq  2 C_\rho \max_i \beta_i r_d(m) +\max_i \sqrt{\frac{\beta_i}{\alpha_i}}\sum_{i=1}^K C_{\mu_i} r_{d_i}(n_i).
\end{equation}
Consequently, for sufficiently large $m, n_1, \dots, n_K$, there is a constant $C > 0$ independent of the sample sizes such that
\[
    \mathbb E W_2(\mathcal{S}(\rho; \bm{\mu}), \mathcal{S}(\rho^m; \bm{\mu}^{\bm{n}})) \leq  C \left\{  r_d(m) + \sum_{i=1}^K r_{d_i}(n_i) \right\}.
\]
\end{enumerate}
\end{corollary}

\begin{remark}
If $d_i>4$ for all $i$, then
\begin{equation*}
    \mathbb E W_2(\mathcal{S}(\rho; \bm{\mu}), \mathcal{S}(\rho^m; \bm{\mu}^{\bm{n}})) \leq 
    \begin{cases}
        C\left(m^{-\frac{1}{d}}+\sum_{i=1}^K n_i^{-\frac{1}{2d_i}}\right) &\text{ if $\varphi_i$'s are $\beta_i$-smooth only},\\
        C\left(m^{-\frac{1}{d}}+\sum_{i=1}^K n_i^{-\frac{1}{d_i}}\right) &\text{ if $\varphi_i$'s are $\alpha_i$-strongly convex and $\beta_i$-smooth}.
    \end{cases}    
\end{equation*}
For fixed $K$, smoothness constants, and moment bounds, sufficiently
large implicit constants in
\[
    m\gtrsim\varepsilon^{-d}, \quad n_i\gtrsim\varepsilon^{-2d_i},\quad 1\leq i\leq K,
\]
ensure $\mathbb E W_2(\mathcal{S}(\rho; \bm{\mu}), \mathcal{S}(\rho^m; \bm{\mu}^{\bm{n}}))\leq\varepsilon$ for sufficiently small $\varepsilon>0$. Therefore, for example, if $m=n_1=\cdots=n_K \asymp n)$ and $d_{\max}=\max_i d_i$, then
\[
    \mathbb E W_2(\mathcal{S}(\rho; \bm{\mu}), \mathcal{S}(\rho^m; \bm{\mu}^{\bm{n}})) =
    \begin{cases}
        O \left( n^{-\min\left\{\frac{1}{d}, \frac{1}{2d_{\max}} \right\}} \right) &\text{ if $\varphi_i$'s are $\beta_i$-smooth only},\\
        O \left( n^{- \frac{1}{d} } \right) &\text{ if $\varphi_i$'s are $\alpha_i$-strongly convex and $\beta_i$-smooth}.
    \end{cases}
\]    
\end{remark}

\begin{corollary}[Concentration of empirical shadow]
\label{cor:hp-shadow-complexity}
Assume the setting of \Cref{thm:smooth-shadow-stability} and $\rho$ and $\bm{\mu}$ have finite moments of order $q$ for some $q > 4$. Then, there exist constants $C,c>0$, independent of the sample sizes
and $\varepsilon$, such that, for $0<\varepsilon\leq1$,
\begin{equation}\label{eq:hp-shadow-tail}
\begin{aligned}
    &\mathbb P\left( W_2(\mathcal{S}(\rho; \bm{\mu}), \mathcal{S}(\rho^m; \bm{\mu}^{\bm{n}}))>\varepsilon \right) \\
    &\leq C\Biggl\{  e^{-cmh_d(\varepsilon^2)}
 +m^{1-\gamma}\varepsilon^{-2\gamma} +\sum_{i=1}^K\left[
 e^{-cn_ih_{d_i}(\varepsilon^4)}
 +n_i^{1-\gamma}\varepsilon^{-4\gamma}
\right]\Biggr\}.
\end{aligned}
\end{equation}

Fix $0 < \varepsilon, \delta < 1$. Then, there is a constant $C_*>0$, independent of $m,n_1,\ldots,n_K,\varepsilon,\delta$, such that if
\begin{equation}\label{eq:hp-sample-sizes}
\begin{aligned}
    m&\geq C_*\left[ \log\frac{2(K+1)}{\delta} \frac{1}{h_d(\varepsilon^2)} +\left(\frac{K+1}{\delta}\right)^{\frac{1}{\gamma - 1}} \varepsilon^{-\frac{2\gamma}{\gamma-1}} \right],\\
    n_i &\geq C_*\left[ \log\frac{2(K+1)}{\delta} \frac{1}{h_{d_i}(\varepsilon^4)} +\left(\frac{K+1}{\delta}\right)^{\frac{1}{\gamma - 1}} \varepsilon^{-\frac{4\gamma}{\gamma-1}} \right], \quad i=1,\ldots,K,
\end{aligned}
\end{equation}
then, it holds that
\begin{equation}\label{eq:hp-confidence-guarantee}
    \mathbb P \left( W_2(\mathcal{S}(\rho; \bm{\mu}), \mathcal{S}(\rho^m; \bm{\mu}^{\bm{n}})) \leq\varepsilon\right)\geq1-\delta.
\end{equation}
The constants may depend on $K$, the dimensions, $q$, $\zeta$,
the population $q$-th moments, and the smoothness constants.
\end{corollary}

\begin{remark}
If every $d_i>4$, then $d>4$, and sufficient sample sizes are
\begin{equation*}\label{eq:hp-high-dimensional-sizes}
\begin{aligned}
m&\geq C_*\left[
\varepsilon^{-d}\log\frac{2(K+1)}{\delta}
+\varepsilon^{-\frac{2\gamma}{\gamma-1}}
\left(\frac{K+1}{\delta}\right)^{\frac{1}{\gamma-1}}
\right],\\
n_i&\geq C_*\left[
\varepsilon^{-2d_i}\log\frac{2(K+1)}{\delta}
+\varepsilon^{-\frac{4\gamma}{\gamma-1}}
\left(\frac{K+1}{\delta}\right)^{\frac{1}{\gamma-1}}
\right].
\end{aligned}
\end{equation*}
For fixed confidence $0 < \delta < 1$ and fixed model constants, $\gamma>2$ implies $\frac{2\gamma}{\gamma-1}<4$ and $\frac{2\gamma}{\gamma-1}<8$. Consequently, as $\varepsilon\downarrow0$, the sufficient orders reduce to
\[
m\gtrsim\varepsilon^{-d},\quad
n_i\gtrsim\varepsilon^{-2d_i}.
\]
If $m=n_1=\cdots=n_K=N$ and $d_{\max}=\max_i d_i$, it is sufficient that
\[
    N\geq C_*\left[
 \varepsilon^{-\max\{d,2d_{\max}\}}
\log\frac{2(K+1)}{\delta}
+ \varepsilon^{-\frac{2\gamma}{\gamma-1}}
\left(\frac{K+1}{\delta}\right)^{\frac{1}{\gamma-1}}
\right].
\]    
Then, \eqref{eq:hp-confidence-guarantee} is obtained.
\end{remark}

\subsection{Uniqueness}
A next theorem is about the uniqueness of projection. It is stated in \cite[Remark 4.2]{EcksteinNutz2022} that neither a shadow nor a projection solving \eqref{problem: Wasserstein projection} is unique in general. However, under absolutely continuous source marginals and regular cost functions, the shadow projection becomes a unique projection for \eqref{problem: Wasserstein projection}. \citet{gangbo1996geometry} introduced a large family of cost functions to study the geometry of optimal transport, which satisfies strict convexity, cone geometry, and superlinearity. Such family includes the $p$-th power of distance for $1 < p < \infty$.

\begin{assumption}\label{assumption: strict convexity of cost function}
Let $c(x,x')=h(x-x')$, where $h: \mathbb{R}^d \rightarrow [0, \infty)$, $h\in C^1$ with $h(0)=0$ is a symmetric function satisfying
\begin{enumerate}
	\item[(A1)] $h$ is strictly convex on $\mathbb{R}^d$,
	\item[(A2)] given a height $r\in \mathbb{R}_+$ and an angle $\theta \in (0,\pi) $, there exists some $M:=M(r, \theta)>0$ such that for all $|p |>M$, one can find a cone 
	\begin{align*}
		K(r, \theta, z,p):=\left\lbrace x \in \mathbb{R}^d  : | x-p|| z|\cos(\theta/2)\leq \left< z,x-p \right>\leq  r| z| \right\rbrace,
	\end{align*}
	with vertex at $p$ on which $h$ attains its maximum at $p$,
	\item[(A3)] $\lim_{|x | \rightarrow \infty}\frac{h(x)}{|x | }= \infty $.
\end{enumerate}    
\end{assumption}

For $\bm{x}, \bm{y} \in \bm{\mathcal{X}}=\prod_{i=1}^K \mathcal{X}_i$, define the separable cost $\bm{c}(\bm{x}, \bm{y}) := \sum_{i=1}^K c_i(x_i, y_i)$. Consider the following optimal transport problems associated with $c_i$ and $\bm{c}$, respectively: for $\mu,\nu \in \mathcal{P}(\mathcal{X}_i)$,
\[
    C_i(\mu, \nu) := \min_{\pi \in \Pi(\mu, \nu)} \int_{\mathcal{X}_i \times \mathcal{X}_i} c_i(x_i, y_i) d\pi(x_i, y_i),
\]
and for $\rho, \xi \in \mathcal{P}_2(\bm{\mathcal{X}})$,
\[
   \bm{C}(\rho, \xi):= \min_{\pi \in \Pi(\rho, \xi)} \int_{\bm{\mathcal{X}} \times \bm{\mathcal{X}}} \bm{c}(\bm{x}, \bm{y}) d\pi(\bm{x}, \bm{y}).
\]
Given $\rho \in \mathcal{P}(\bm{\mathcal{X}})$ and $\bm{\mu} \in \prod_{i=1}^K \mathcal{P}(\mathcal{X}_i)$, the optimal transport projection problem associated with $\bm{c}$ is defined as 
\begin{equation}\label{problem: optimal transport projection}
    \min_{\pi \in \Pi(\bm{\mu})} \bm{C}(\pi, \rho).
\end{equation}
Notice that for any $\pi \in \Pi(\bm{\mu})$, one has
\[
    \bm{C}(\pi, \rho) \geq \sum_{i=1}^K C_i(\pi_i, \rho_i) = \sum_{i=1}^K C_i(\mu_i, \rho_i).
\]
We denote the shadow (even if multiple shadows exist possibly) for \eqref{problem: optimal transport projection} by $\mathcal{S}_{\bm{c}}(\rho; \bm{\mu})$. For a general separable cost $\bm{c}$, $\mathcal{S}_{\bm{c}}(\rho; \bm{\mu})$ is a feasible solution for \eqref{problem: optimal transport projection}. The following proposition states that it is indeed a unique solution.

\begin{theorem}[Uniqueness]\label{thm:uniqueness of projection}
Let $\rho\in\mathcal{P}(\bm{\mathcal{X}})$ be any joint law with marginals $\rho_i$ such that each $\rho_i \in \mathcal{P}^{ac}(\mathcal{X}_i)$. Assume that each $c_i(x,x')=h_i(x-x')$ satisfies Assumption~\ref{assumption: strict convexity of cost function}, and
\begin{equation}\label{eq: moment condition for h}
    C_i(\rho_i, \mu_i) <\infty.
\end{equation}
Then, \eqref{problem: optimal transport projection} has a unique projection. Indeed the shadow projection is a unique projection for \eqref{problem: optimal transport projection}. Furthermore,
\[
    \min_{\pi \in \Pi(\bm{\mu})} \bm{C}(\pi, \rho) = \sum_{i=1}^K C_i(\rho_i, \mu_i).
\]
\end{theorem}

\section{Application to Transfer Learning}
Transfer learning uses information about the covariate-response relationship in a source domain to predict responses in a target domain. If covariates and responses share many similarities, it is possible to improve performance in the target by applying knowledge gained from the source for the target.

Let $(X,Y)$ be a vector of covariate and response in a source domain. Assume that there is some function $F$ such that $Y= F(X)$. We write $X \sim \rho_1, Y \sim \rho_2:=F_\#(\rho_1)$. Hence, $(X,Y) \sim \rho:=(\mathrm{Id}, F)_\# (\rho_1)$, which is the source measure. It is assumed that $\rho_1$ and $\rho_2$ are absolutely continuous and have the finite second moments.

Consider covariate and response vectors in a target domain, denoted by $U$ and $V$, respectively. We assume that the \emph{joint distribution $(U,V)$ is unknown}, and \emph{only marginal distributions are accessible}. That is, $U$ and $V$ are assumed to be accessible separately only. Let $\mu_1$ and $\mu_2$ denote the distributions of $U$ and $V$, respectively, and set $\bm{\mu}:=(\mu_1, \mu_2)$. The goal is to \emph{construct a joint distribution (coupling) between $\mu_1$ and $\mu_2$ in the target domain using $\rho$ in the source domain}.

A natural way to transfer $F$ from source to target domain is to use optimal transport. Assume that there are Brenier potentials transporting from $\varphi_1$ and $\varphi_2$ $\rho_1$ to $\mu_1$ and $\rho_2$ to $\mu_2$, respectively. Let $T_1:=\nabla \varphi_1$ and $T_2:=\nabla \varphi_2$. By the definition, we have
\[
    U= T_1(X), \quad V=T_2(Y).
\]
Assuming $\mu_1 \in \mathcal{P}_2^{ac}(\mathcal{X}_1)$ further, \Cref{thm: Brenier's thm} guarantees the existence of the inverse of $T_1$ $\mu_1$-almost everywhere sense. Therefore, there is a map $G: \mathcal{X}_1 \rightarrow \mathcal{X}_2$ such that
\begin{equation}\label{eq: construction of G}
    V= G(U), \quad G:=T_2 \circ F \circ T_1^{-1}.
\end{equation}
The following commutative diagram illustrates the scheme of this transfer learning framework:
\begin{equation*}\label{eq: diagram}
    \newcommand{\cir}{\rotatebox[origin=c]{180}{$\circlearrowright$}}
\begin{tikzcd}
    X \arrow{r}{F}  \rotatebox[origin=c]{180}{\drar[phantom, "\cir"]} & Y \arrow{d}{T_2} \\
    U \arrow{r}{G} \arrow{u}{T_1^{-1}} & V .
\end{tikzcd}    
\end{equation*}
It is not hard to see that the law of $(U, G(U))$, a coupling between the laws of $U$ and $V$, is the shadow of $\rho$ onto $\Pi(\bm{\mu})$. Let $\mathcal{S}( \rho; \bm{\mu})$ denote this shadow. \Cref{thm:uniqueness of projection} implies that for fixed F, the induced target joint law is unique, and $G$ is determined $\mu_1$-almost everywhere.

In reality, only samples from these distributions are observed. Hence, the goal is to construct an estimator $\hat{G}: \mathcal{X}_1 \to \mathcal{X}_2$ based on the samples, and to control its error compared to $G$. Let $(X_1, Y_1), \dots, (X_m, Y_m)$ be $m$ i.i.d. samples from $\rho=(\mathrm{Id}, F)_\#(\rho_1)$, and $U_1, \dots, U_{n_1}$ and $V_1, \dots, V_{n_2}$ be i.i.d. samples from $\mu_1$ and $\mu_2$, respectively. $\rho^m$, $\mu_1^{n_1}$ and $\mu_2^{n_2}$ denote the empirical distributions of $(X_k, Y_k)$, $U_j$ and $V_j$, respectively. Also, we use $\rho_1^m$ and $\rho_2^m$ to denote the marginal distributions of $\rho^m$, and set $\bm{\mu}^{\bm{n}}:= (\mu_1^{n_1}, \mu_2^{n_2})$.

Suppose there is an estimator $\hat{F}$ such that
\[
    Y_k=\hat{F}(X_k) \text{ for all $k=1, \dots, m$.}
\]
In words, $\hat{F}$ perfectly interpolates the source sample, which is commonly observed in modern machine learning. For $i=1,2$, define
\[
    \Phi^i_{\alpha, \beta} := \left\{ \varphi \in C^1(\mathbb{R}^{d_i}) :  \alpha \mathrm{Id} \preceq D^2 \varphi \preceq \beta \mathrm{Id}, \, 0 \leq \alpha \leq \beta <\infty\right\}
\]
Since $\rho_i^m$ and $\mu_i^{n_i}$'s are discrete, there might be no transport maps from $\rho_i^m$ to $\mu_i^{n_i}$. Thus, it is necessary to consider the estimators, $\hat{\varphi}_1$ and $\hat{\varphi}_2$, which are obtained from the following Monge-type optimization problems:
\begin{equation}\label{eq: minimization over strongly convex/smooth functions}
    \hat{\varphi}_i \in \argmin_{\hat{\varphi} \in \Phi^i_{\alpha, \beta} } W_2((\nabla \hat{\varphi})_\#(\rho_i^m), \mu_i^{n_i}).
\end{equation}
While it is not a priori clear \eqref{eq: minimization over strongly convex/smooth functions} has a solution, \citet{Constrained_approximate_optimal_transport_maps} recently studied this problem, and showed that it has a solution.

\begin{lemma}\cite{Constrained_approximate_optimal_transport_maps}[Theorem 2.4 and Proposition 2.9]\label{lem: existence of regular solution}
Consider
\[
    \inf_{g \in \mathcal{G}} T_c(g_\#(\mu), \nu), \quad T_c(g_\#(\mu), \nu):= \min_{\pi \in \Pi(\mu, \nu)} \int c(g(x), y) \pi(dx, dy).
\]
Let $c: \mathbb{R}^d \times \mathbb{R}^d \longrightarrow \mathbb{R}_+$ be a lower semi-continuous cost function, $\mu, \nu \in \mathcal{P}(\mathbb{R}^d)$. Assume that
\begin{enumerate}
    \item[(i)] (Coercive cost) There exists $\eta : \mathbb{R}_+ \longrightarrow \mathbb{R}_+$ non-decreasing and such that $\lim_{t \to \infty} \eta(t) = \infty$ and $\int \eta(\|v\|_2) \nu(dv) < \infty$, and $\alpha \in \mathbb{R}$ such that for any $u, v \in \mathbb{R}^d$, $c(u,v) \geq \alpha + \eta (\|u - v\|_2)$.
    \item[(ii)] (Gradients of convex functions) Let $\{E_1, \dots, E_k\}$ be a partition of $\mathbb{R}^d$ such that each $E_i$ is arc-connected. Then for $0 \leq \alpha \leq \beta < \infty$
    \[
        \mathcal{G}=\{ g : \text{$g$ is $\beta$-Lipschitz, } g |_{\text{int}(E_i)} = \nabla \varphi_i \text{ for some $\alpha$-strongly convex $\varphi_i$} \}.
    \]
    \item[(iii)] (Problem finiteness) There exists $g \in \mathcal{G}$ such that $T_c(g_\#(\mu), \nu) < \infty$.
\end{enumerate}    
Then, the problem $\inf_{g \in \mathcal{G}} T_c(g_\#(\mu), \nu)$ has a solution.
\end{lemma}

Let $\hat{T}_i:= \nabla \varphi_i$ where $\varphi_i$'s are solutions for \eqref{eq: minimization over strongly convex/smooth functions} with $\alpha=0$ and $\beta < \infty$. Since $\mu_i^{n_i}$'s are not absolutely continuous, $\hat{T}_i$'s need not be invertible. Thus, we will estimate $\hat{G}$ by empirical risk minimization over Lipschitz class. Fix $L>0$, define
\begin{equation}
\label{eq:lipschitz-approximation-class}
\begin{aligned}
\mathcal G_L =
\left\{
g:\mathcal{X}_1 \rightarrow \mathcal{X}_2:
\operatorname{Lip}(g)\leq L
\right\}, \quad A_L(G) =
\inf_{g\in\mathcal G_L}
\|G-g\|_{L^2(\mu_1)}.
\end{aligned}
\end{equation}
$A_L(G)$ is the approximation error of $G$ by the Lipschitz function class.

Notice that since the target observations $U_j$ and $V_j$ are unpaired, a standard empirical risk minimization framework is not applicable. However, leveraging $\hat{T}_1$ and $\hat{T}_2$, one can formulate the empirical risk minimization framework \emph{based on paired data transported} from $(X,Y)$ to $(U, V)$ as follows:
\begin{equation}
\label{eq:growing-class-estimator}
\widehat G_L
\in
\operatorname*{argmin}_{g\in\mathcal G_L}
\frac1m\sum_{i=1}^m
\left\|
\widehat T_2(Y_i)-g(\widehat T_1(X_i))
\right\|^2.
\end{equation}
In words, by transferring the paired data $(X_i, Y_i)$ to $(\widehat T_1(X_i), T_2(Y_i))$ we can define a supervised empirical risk minimization problem.

Now, the goal is to estimate the generalization error of the transfer learning. The $\ell^2$ generalization error of $\hat{G}$ is defined as
\[
    \| G - \hat{G}_L \|_{L^2(\mu_1)}.
\]
We achieve an upper bound of the generalization error of the transfer learning by the approximation error $A_L(G)$ and the discrepancy between the shadow and its counterpart. Furthermore, assuming the consistency of $\hat{F}$, one can show the consistency of $\hat{G}_L$.

\begin{proposition}[Generalization error]\label{prop:approximation-shadow-decomposition}
Let $F : \mathcal{X}_1 \rightarrow \mathcal{X}_2$ be a fixed function such that $\rho = (\mathrm{Id}, F)_\#(\rho_1)$, i.e., $\rho_2 = F_\#(\rho_1)$. Assume that $\mathcal{X}_1$ is arc-connected, and $\rho_1, \mu_1 \in \mathcal{P}_2^{ac}(\mathcal{X}_1)$, $\rho_2 \in \mathcal{P}_2^{ac}(\mathcal{X}_2)$ and $\mu_2 \in \mathcal{P}_2(\mathcal{X}_2)$. For each $i=1,2$, let $\hat\varphi_i$ be a solution for \eqref{eq: minimization over strongly convex/smooth functions} with $\alpha=0$ and $0 < \beta <\infty$. Fix $L>0$, and let
\[
\widehat T_i=\nabla\widehat\varphi_i,
\quad
\widehat{\nu}_i=(\widehat T_i)_\#\rho_i^m,
\quad
\widehat{\bm{\nu}}=(\widehat{\bm{\nu}}_1,\widehat{\bm{\nu}}_2).
\]
Also, assume that there is a minimizer for \eqref{eq:growing-class-estimator} denoted by $\hat{G}_L$. Recall the definition of $G$ from \eqref{eq: construction of G}. Then,
\begin{equation}
\label{eq:approximation-shadow-decomposition}
\|\widehat G_L-G\|_{L^2(\mu_1)}
\leq 
A_L(G) + 2\sqrt{1+L^2} W_2 (\mathcal{S}(\rho;\bm{\mu}), \mathcal{S}(\rho^m; \widehat{\bm{\nu}}) ).
\end{equation}
\end{proposition}

From \eqref{eq:approximation-shadow-decomposition}, one can show the consistency of $\hat{G}_L$ as $L, n_1, n_2 ,m \to \infty$ such that $L$ grows sufficiently slowly.

\begin{proposition}[Consistency]\label{prop: qualitative consistency}
Under the same setting of \Cref{prop:approximation-shadow-decomposition}, assume further that $\{\hat{F}=\hat{F}_m\}$ has a convergent subsequence, and each Brenier potential $\varphi_i$ from $\rho_i$ to $\mu_i$ is $\beta_i$-smooth for some $\beta_i \leq \beta$. Let $L, n_1, n_2 ,m \to \infty$ such that
\[
    L  W_2 (\mathcal{S}(\rho;\bm{\mu}), \mathcal{S}(\rho^m; \widehat{\bm{\nu}}) ) \rightarrow 0
\]
almost surely. Then, $\hat{G}_L$ has a subsequence converging to $G$ in $L^2(\mu_1)$ almost surely.
\end{proposition}

Now, we assume that there are Brenier potentials $\varphi_1, \varphi_2$ from $\rho_1$ to $\mu_1$ and $\rho_2$ to $\mu_2$, respectively, such that both are strongly convex and smooth. This case includes the transfer learning between strongly log-concave measures. In particular, we assume that $0 < \alpha \leq \beta < \infty$ such that each $\varphi_i$ is $\alpha_i$-strongly convex and $\beta_i$-smooth for some $\alpha_i \geq \alpha$ and $\beta_i \leq \beta$.

Under this assumption, the true Brenier potential $\varphi_i$ belongs to $\Phi^i_{\alpha, \beta}$. This implies that, recalling \eqref{eq: construction of G}, $G$ is Lipschitz with $\mathrm{Lip}(G) \leq \frac{\beta}{\alpha} \mathrm{Lip}(F)$, hence $A_{\frac{\beta}{\alpha} \mathrm{Lip}(F)}(G)=0$.

In addition, $\hat{\varphi}_1 \in \Phi^1_{\alpha, \beta}$ has the inverse. Therefore, one can obtain an estimator of $G$ by
\[
    \hat{G}:= (\nabla \hat{\varphi}_2) \circ \hat{F} \circ (\nabla \hat{\varphi}^*_1).
\]
Note that if $\hat{F}$ is Lipschitz, then $\hat{G}$ is Lipschitz as well with $\mathrm{Lip}(\hat{G}) \leq \frac{\hat{\beta}_2}{\hat{\alpha}_1} \mathrm{Lip}(\hat{F})$.

The proof for the following theorem is a direct application of \Cref{lem:root-risk-comparison}: note that
\[
 \int\|v- \hat{G}(u)\|^2 \,dS(\rho^m;\widehat{\bm\nu})(u,v) = \frac{1}{m}\sum_{j=1}^m \left\| \widehat T_2(Y_j)-\hat{G}(\widehat T_1(X_j)) \right\|^2 = 0.
\]
Hence, we omit the proof.

\begin{corollary}\label{thm: generalization error of transfer learning}
Assume that $F$ is Lipschitz, and for each $i=1,2$, a Brenier potential $\varphi_i$ from $\rho_i$ to $\mu_i$ is $\alpha_i$-strongly convex and $\beta_i$-smooth. Assume further that almost surely $\mathrm{Lip}(\hat{F}), \mathrm{Lip}(F) \leq L <\infty$ for every sample realization, and $\alpha_1, \alpha_2 \geq \alpha$ and $\beta_1, \beta_2 \leq \beta$. Let $\hat{G}:= (\nabla \hat{\varphi}_2) \circ \hat{F} \circ (\nabla \hat{\varphi}^*_1)$. Then, $A_{(\beta/\alpha) L}(G)=0$ and $\hat{G}$ is Lipschitz with $\mathrm{Lip}(\hat{G}) \leq \frac{\beta}{\alpha} L$, therefore
\[
    \|\widehat G-G\|_{L^2(\mu_1)}
\leq \sqrt{1+\left( \frac{\beta}{\alpha}L \right)^2} W_2 (\mathcal{S}(\rho;\bm{\mu}), \mathcal{S}(\rho^m; \widehat{\bm{\nu}}) ).
\]
\end{corollary}

Combined with \Cref{lem: rate of convergence of empirical measures}, one can obtain the sample complexity and the concentration of the generalization error of transfer learning. Also, we derive the required PAC learning-type sample sizes.

Recall $r_d(n)$ and $h_d(t)$, defined in \eqref{eq:shadow-empirical-rate} and \eqref{eq:hp-h-function}, respectively.

\begin{theorem}\label{thm: statistics of transfer learning}
Assume the same setting of \Cref{thm: generalization error of transfer learning}. Assume further that $\rho$ and $\bm{\mu}$ have finite moments of order $q$ for some $q > 4$.
\begin{enumerate}
    \item[(i)] (Sample complexity) There is a constant $C$ independent of $n,m$ such that
\begin{align*}
    \mathbb{E}\|G - \widehat{G}\|_{L^2(\mu_1)} \leq C \left\{ r_{d_1}(m) + r_{d_1}(n_1) + r_{d_2}(n_2)\right\}.
\end{align*}
In particular, if $d_1, d_2 >4$, then
\[
    \mathbb{E}\|G - \widehat{G}\|_{L^2(\mu_1)} \leq C \left\{ m^{-\frac{1}{d_1}} + n_1^{-\frac{1}{d_1}} + n_2^{-\frac{1}{d_2}} \right\}.
\]
    \item[(ii)] (Concentration) Let $\zeta\in(0,q-4)$, and put $\gamma:=\frac{q-\zeta}{2}>2$. There exist constants $C,c>0$, independent of $m,n_1,n_2,t$, such that, for every $0<t\leq1$,
\[
    \mathbb P\!\left( \frac{\|G-\widehat G\|_{L^2(\mu_1)}}  {\sqrt{1+\left( \frac{\beta}{\alpha}L \right)^2}} >Ct\right) \leq  C\left( e^{-cmh_{d_1}(t^2)} + m^{1-\gamma}t^{-2\gamma} + \sum_{i=1}^2 \left( e^{-cn_ih_{d_i}(t^2)}  +  n_i^{1-\gamma}t^{-2\gamma}\right) \right).
\]
The constants may depend on $q,r,d_1,d_2$, the population $q$-th moments, Lipschitz constant $L$ and $\alpha, \beta$.
    \item[(iii)] (Sample sizes) For $\varepsilon,\delta\in(0,1)$, set $t:= \left( C \sqrt{1+\left( \frac{\beta}{\alpha}L \right)^2}\right)^{-1} \varepsilon$. Choose positive integers $m,n_1,n_2$ satisfying
\begin{equation}
\label{eq:transfer-n-m-choice}
\begin{aligned}
m\geq\max\Biggl\{&
\frac{\log(3C/\delta)}{c h_{d_1}(t^2)}, \left(\frac{3C}{\delta}\right)^{\frac{1}{\gamma-1}}
 t^{-\frac{2\gamma}{\gamma-1}}
\Biggr\}, \quad n_i\geq\max\left\{
\frac{\log(3C/\delta)}{c\,h_{d_i}(t^2)},
\left(\frac{3C}{\delta}\right)^{\frac{1}{\gamma-1}}
 t^{-\frac{2\gamma}{\gamma-1}}
\right\} \text{ for $i=1,2$}.
\end{aligned}
\end{equation}
Then,
\[
    \mathbb P\!\left( \|G-\widehat G\|_{L^2(\mu_1)}\leq\varepsilon \right)\geq1-\delta.
\]
\end{enumerate}
\end{theorem}

\begin{remark}
Assume that $d_1,d_2>4$. Sufficient conditions are
\[
m\geq C\left[
\varepsilon^{-2d_1}
\log\frac{1}{\delta}
+\delta^{-\frac1{\gamma-1}}
\varepsilon^{-\frac{2\gamma}{\gamma-1}}
\right], \quad
n_i\geq C\left[
\varepsilon^{-2d_i}\log\frac{1}{\delta}
+\delta^{-\frac1{\gamma-1}}
\varepsilon^{-\frac{2\gamma}{\gamma-1}}
\right]
\text{ for $i=1,2$}.
\]
In other words, for fixed $\delta$ and $\gamma>2$, it is sufficient to choose $m, n_1$ and $n_2$ as
\[
m\gtrsim\varepsilon^{-2 d_1},
\quad
n_i\gtrsim\varepsilon^{-2d_i}
\text{ for $i=1,2$}.
\]    
\end{remark}

\section{Proof for stability and uniqueness of shadow}

\subsection{Proof of Theorem~\ref{thm:smooth-shadow-stability}}
We first start to prove the following lemma regarding the stability bounds for optimal transport maps. The first proves \(\frac{1}{2}\)-Hölder continuity of transport map under the assumption that it is Lipschitz, which was studied by \citet{Li_Nochetto_2021}. The second one is \cite[Theorem 3.1]{balakrishnan2025stabilityboundssmoothoptimal}, which shows the Lipschitz stability with the multiplicative constant proportion to the condition number (the ratio between Lipschitz and convexity constants).

\begin{lemma}\label{lem:smooth-coupling-stability}
Let $\mu,\nu,\widetilde \mu,\widetilde \nu\in\mathcal{P}_2(\mathbb{R}^d)$. Let $\varphi \in C^1$ be a Brenier potential from $\mu$ to $\nu$
\begin{enumerate}
    \item[(i)] Assume that $\varphi$ is $\beta$-smooth, i.e., $T=\nabla \varphi$ is $\beta$-Lipschitz for some $0<\beta<\infty$. Then, for any $\pi \in \Pi_{\mathrm{opt}}(\widetilde \mu,\widetilde \nu)$, it holds that
\begin{equation}\label{eq:smooth-coupling-stability}
    \left( \mathbb{E}_{(X,Y) \sim \pi} \|Y-T(X)\|^2 \right)^{\frac{1}{2}} \leq W_2(\nu,\widetilde \nu) + \sqrt{\beta^2 W^2_2(\mu,\widetilde \mu) + 4 \beta s(\widetilde{\mu}) W_2(\nu,\widetilde \nu)}.
\end{equation}
    \item[(ii)] (\cite[Theorem 3.1]{balakrishnan2025stabilityboundssmoothoptimal}) Assume that $\varphi$ is $\alpha$-strongly convex (i.e., $D^2 \varphi \succeq \alpha \text{I}$) and $\beta$-smooth for some $0 <\alpha \leq \beta < \infty$. Then, for any $\pi \in \Pi_{\mathrm{opt}}(\widetilde \mu,\widetilde \nu)$, it holds that 
\begin{equation}\label{eq:smooth-convex-coupling-stability}
    \left( \mathbb{E}_{(X,Y) \sim \pi} \|Y-T(X)\|^2 \right)^{\frac{1}{2}} \leq  \beta W_2(\mu, \widetilde \mu) + \sqrt{\frac{\beta}{\alpha}} W_2(\nu, \widetilde \nu) .
\end{equation}
\end{enumerate}
\end{lemma}

\begin{proof}
Here, we focus on the proof of (i). Recall the Fenchel gap 
\[
    \mathcal F_\varphi(x,z) =\varphi(x)+\varphi^*(z)-\langle x,z\rangle.
\]
Fenchel's inequality gives $\mathcal F_\varphi(x,z) \geq 0$ and the equality holds as follows: for $T(u)=\nabla \varphi(u)$,
\begin{equation}\label{eq:fenchel-equality}
    \varphi(u)+\varphi^*(T(u))=\langle u,T(u)\rangle.
\end{equation}
Indeed, the supporting-hyperplane inequality implies
\[
    \langle w,T(u)\rangle-\varphi(w) \leq\langle u,T(u)\rangle-\varphi(u),
\]
with equality at $w=u$.

Define the Bregman divergence associated with $\varphi$ by
\begin{equation}\label{eq:bregman-definition}
    D_\varphi(v,u) =\varphi(v)-\varphi(u)-\langle T(u),v-u\rangle.
\end{equation}
Convexity of $\varphi$ implies $D_\varphi(v,u)\geq0$. Moreover, setting $g=v-u$, $\beta$-Lipschitzness of $T$ leads to
\begin{align}\label{eq:bregman-smoothness}
    D_\varphi(v,u)
    &=\int_0^1 \langle T(u+th)-T(u),g\rangle\,dt \notag\\
    &\leq\int_0^1\beta t\|g\|^2\,dt =\frac{\beta}{2}\|v-u\|^2.
\end{align}
Combining \eqref{eq:fenchel-equality} with \eqref{eq:bregman-definition} yields the identity
\begin{equation*}\label{eq:fenchel-bregman}
    \mathcal F_\varphi(v,T(u))=D_\varphi(v,u).
\end{equation*}

By using \eqref{eq:bregman-smoothness}, one can obtain
\[
    \varphi(x+g) \leq\varphi(x)+\langle T(\bm{x}),g\rangle +\frac{\beta}{2}\|g\|^2.
\]
Thus, for every $g\in\mathbb{R}^d$,
\begin{align*}
    \varphi^*(z) \geq\langle x+g,z\rangle-\varphi(x+g) \geq\langle x,z\rangle-\varphi(x) +\langle z-T(\bm{x}),g\rangle-\frac{\beta}{2}\|g\|^2.
\end{align*}
Taking $g=\frac{z-T(\bm{x})}{\beta}$ results in a lower bound on the Fenchel gap as
\begin{equation}\label{eq:fenchel-coercivity}
    \mathcal F_\varphi(x,z) \geq\frac{1}{2\beta}\|z-T(\bm{x})\|^2.
\end{equation}

Define
\[
 \Delta_0 =\min_{\gamma\in\Pi(\widetilde \mu,\nu)} \int\mathcal F_\varphi(x,z)\,\gamma(dx,dz).
\]
The terms involving $\varphi$ and $\varphi^*$ depend only on the
marginals. Thus minimizing this integral is equivalent to maximizing
$\int\langle x,z\rangle\,d\gamma$, and hence to minimizing
$\int\|x-z\|^2\,d\gamma$. In particular, any
$\gamma_0\in\Pi_{\mathrm{opt}}(\widetilde \mu,\nu)$ attains $\Delta_0$.

Choose $(V,U)$ with law in $\Pi_{\mathrm{opt}}(\widetilde \mu,\mu)$.
Since $T_\#\mu=\nu$, the law of $(V,T(U))$ belongs to
$\Pi(\widetilde \mu,\nu)$. It follows from this coupling that
\begin{align}
    0\leq\Delta_0
    &\leq\E\bigl[ \varphi(V)+\varphi^*(T(U))-\langle V,T(U)\rangle \bigr] \notag\\
    &=\E\bigl[ \varphi(V)-\varphi(U)-\langle T(U),V-U\rangle \bigr] \notag\\
    &=\E D_\varphi(V,U) \notag\\
    &\leq\frac{\beta}{2}\mathbb{E}\|V-U\|^2 =\frac{\beta}{2}W_2^2(\widetilde{\mu}, \mu).
\label{eq:reference-gap}
\end{align}
Here the first equality uses Fenchel equality
\eqref{eq:fenchel-equality}, and the last inequality uses
\eqref{eq:bregman-smoothness}.

Let $(X,Y)\sim\pi \in \Pi_{\mathrm{opt}}(\widetilde \mu,\widetilde \nu)$. Glue this coupling with an optimal coupling
of $\widetilde \nu$ and $\nu$ to obtain $(X,Y,Z)$ such that
\[
    (X,Y)\sim\pi, \quad \mathcal L(Y,Z)\in\Pi_{\mathrm{opt}}(\widetilde \nu,\nu), \quad \mathbb{E}\|Y-Z\|^2=W_2^2(\nu, \widetilde{\nu}).
\]
Fix $\gamma_0\in\Pi_{\mathrm{opt}}(\widetilde \mu,\nu)$ and disintegrate it as
\[
\gamma_0(dx_0,dz)=k_0(dx_0\mid z)\,\nu(dz).
\]
Extend the preceding joint distribution by drawing $X_0$
conditionally on $(X,Y,Z)$ according to $k_0(\cdot\mid Z)$.
Then $(X_0,Z)\sim\gamma_0$, while the law of $(X,Y,Z)$ is unchanged.

Since $(X_0,Y)$ is an admissible coupling of
$\widetilde \mu$ and $\widetilde \nu$, optimality of $(X,Y)$ gives
\begin{equation}\label{eq:optimality-inner-product}
    \mathbb{E} \langle X_0,Y\rangle\leq\mathbb{E}\langle X,Y\rangle.
\end{equation}
Both $X_0$ and $X$ have the same law $\widetilde \mu$, so
\begin{align*}
    (\mathbb{E}\|X_0-X\|^2)^{\frac{1}{2}} 
    &\leq (\mathbb{E}\|X_0- \mathbb{E}_{\widetilde{\mu}}[X]\|^2)^{\frac{1}{2}} +(\mathbb{E}\|X - \mathbb{E}_{\widetilde{\mu}}[X]\|^2)^{\frac{1}{2}}\\
    &=2\sqrt{\mathrm{tr}\,\mathrm{Cov}_{\widetilde{\mu}}(X)}\\
    &=2s(\widetilde{\mu}).
\end{align*}
Using the common marginal laws, \eqref{eq:optimality-inner-product},
and Cauchy--Schwarz inequality, we obtain
\begin{align}
    \mathbb{E}\mathcal F_\varphi(X,Z)-\Delta_0
    &=\mathbb{E}\mathcal F_\varphi(X,Z) -\mathbb{E}\mathcal F_\varphi(X_0,Z) \notag\\
    &=\mathbb{E}\langle X_0-X,Z\rangle \notag\\
    &=\mathbb{E}\langle X_0-X,Z-Y\rangle +\mathbb{E}\langle X_0-X,Y\rangle \notag\\
    &\leq\mathbb{E}\langle X_0-X,Z-Y\rangle \notag\\
    &\leq (\mathbb{E}\|X_0-X\|^2)^{\frac{1}{2}} (\mathbb{E}\|Z-Y\|^2)^{\frac{1}{2}} \leq 2 s(\widetilde{\mu}) W_2(\nu, \widetilde{\nu}).
\label{eq:target-gap}
\end{align}
Combining \eqref{eq:fenchel-coercivity},
\eqref{eq:reference-gap}, and \eqref{eq:target-gap} gives
\[
    \frac{1}{2\beta}\mathbb{E}\|Z-T(\bm{x})\|^2 \leq\mathbb{E}\mathcal F_\varphi(X,Z) \leq \frac{\beta}{2}W_2^2(\widetilde{\mu}, \mu) + 2s(\widetilde{\mu}) W_2(\nu, \widetilde{\nu}).
\]
Lastly, Minkowski's inequality yields
\begin{align*}
    (\mathbb{E}\|Y-T(\bm{x})\|^2)^{\frac{1}{2}} 
    &\leq (\mathbb{E}\|Y-Z\|^2)^{\frac{1}{2}} +(\mathbb{E}\|Z-T(\bm{x})\|^2)^{\frac{1}{2}}\\
    &\leq W_2(\nu, \widetilde{\nu})+\sqrt{\beta^2 W_2^2(\widetilde{\mu}, \mu)+4\beta s(\widetilde{\mu}) W_2(\nu, \widetilde{\nu})},
\end{align*}
which proves \eqref{eq:smooth-coupling-stability}.
\end{proof}

\begin{proof}[Proof of \Cref{thm:smooth-shadow-stability}]
(i) Let $T(\bm{x}):=(T_1(x_1),\ldots,T_K(x_K))$ such that $T_i= \nabla \varphi_i$ for each $i=1,\dots, K$. Applying \Cref{lem:smooth-coupling-stability} with
$\mu=\widetilde \mu=\rho_i$ and $\nu=\widetilde \nu=\mu_i$ shows that every optimal marginal coupling is concentrated on $y_i=T_i(x_i)$. Therefore the reference shadow is unique and
\[
    \mathcal{S}(\rho; \bm{\mu})=T_\#\rho.
\]
Moreover,
\[
    \|T(\bm{x})-T(x')\|^2=\sum_{i=1}^K\|T_i(x_i)-T_i(x_i')\|^2 \leq \max_i \beta_i^2\|x-x'\|^2.
\]
Pushing forward an optimal coupling of $\rho$ and $\xi$ by
$(T,T)$ yields
\[
    W_2(T_\#\rho,T_\#\xi)\leq\max_i \beta_i W_2(\rho, \xi).
\]
Then the triangle inequality gives
\begin{equation}\label{eq:shadow-triangle}
    W_2(\mathcal{S}(\rho; \bm{\mu}),S(\xi; \bm{\nu})) \leq\max_i \beta_i W_2(\rho, \xi)+ W_2(T_\#\xi,S(\xi; \bm{\nu})).
\end{equation}

Choose the marginal optimal couplings defining the selected shadow:
\[
\pi_i(dx_i,dy_i)
=\xi_i(dx_i)\,\kappa_i(dy_i\mid x_i)
\in\Pi_{\mathrm{opt}}(\xi_i,\nu_i),
\]
and let $(X,Y)$ have joint law
\[
    \pi(d \bm{x}, d \bm{y})=\xi(d \bm{x})\prod_{i=1}^K\kappa_i(dy_i\mid x_i).
\]
Then $X\sim\xi,$ $Y\sim S(\xi; \bm{\nu})$ and $(X_i,Y_i)\sim\pi_i$. In particular, $(T(X),Y)$ is an admissible coupling of
$T_\#\xi$ and $S(\xi; \bm{\nu})$.

For each $i$, by applying \eqref{eq:smooth-coupling-stability} of \Cref{lem:smooth-coupling-stability} we obtain
\[
    \left(\mathbb{E}_{(X_i,Y_i) \sim \pi_i}\|Y_i-T_i(X_i)\|^2 \right)^{\frac{1}{2}} \leq W_2(\mu_i, \nu_i) + \sqrt{\beta_i^2 W_2^2(\rho_i, \xi_i) + 4 \beta_i s(\xi_i) W_2(\mu_i, \nu_i)}.
\]
Consequently, by the separability of $\| \cdot \|_2^2$, Minkowski's inequality and the subadditivity of the square root, it follows that
\begin{align*}
    W_2(T_\#\xi,S(\xi; \bm{\nu}))
    & \leq \left( \mathbb{E}_{(X,Y) \sim \pi} \|Y-T(X)\|^2\right)^{\frac{1}{2}} \\
    &= \left(\sum_{i=1}^K \mathbb{E}_{(X_i,Y_i) \sim \pi_i} \|Y_i-T_i(X_i)\|^2\right)^{\frac{1}{2}}\\
    &\leq\left(\sum_{i=1}^K( W_2(\mu_i, \nu_i) + \sqrt{\beta_i^2 W_2^2(\rho_i, \xi_i) + 4 \beta_i s(\xi_i) W_2(\mu_i, \nu_i)} )^2\right)^{\frac{1}{2}}\\
    &\leq\left(\sum_{i=1}^K W^2_2(\mu_i, \nu_i) \right)^{\frac{1}{2}} +\left(\sum_{i=1}^K \beta_i^2 W_2^2(\rho_i, \xi_i) + 4 \beta_i s(\xi_i) W_2(\mu_i, \nu_i) \right)^{\frac{1}{2}}\\
    &\leq \bm{W}_2(\bm{\mu}, \bm{\nu}) + \left( \sum_{i=1}^K \beta_i^2 W_2^2(\rho_i, \xi_i) \right)^{\frac{1}{2}} + 2 \left(\sum_{i=1}^K \beta_i s(\xi_i) W_2(\mu_i, \nu_i) \right)^{\frac{1}{2}}\\
    &\leq \bm{W}_2(\bm{\mu}, \bm{\nu}) + \max_i \beta_i
W_2(\rho, \xi) + 2 \left(\sum_{i=1}^K \beta_i s(\xi_i) W_2(\mu_i, \nu_i) \right)^{\frac{1}{2}}.
\end{align*}
where the last line follows by
\begin{equation}\label{eq:marginal-wasserstein-bound}
    \sum_{i=1}^K W_2^2(\rho_i, \xi_i) \leq W_2^2(\rho, \xi).
\end{equation}
Notice that the product Euclidean norm implies
\[
    s(\xi)^2=\sum_{i=1}^K s(\xi_i)^2.
\]
Thus Cauchy--Schwarz inequality gives
\begin{equation*}
    \sum_{i=1}^K \beta_i s(\xi_i)W_2(\mu_i, \nu _i) \leq \max_i \beta_i \left(\sum_{i=1}^K s(\xi_i)^2\right)^{\frac{1}{2}} \left(\sum_{i=1}^K W_2^2(\mu_i, \nu_i)\right)^{\frac{1}{2}} =\max_i \beta_is(\xi)\bm{W}_2(\bm{\mu}, \bm{\nu}).    
\end{equation*}
Thus,
\[
    W_2(T_\#\xi,S(\xi; \bm{\nu})) \leq \max_i \beta_i W_2(\rho, \xi) +  \bm{W}_2(\bm{\mu}, \bm{\nu}) +2 \sqrt{ \max_i \beta_i s(\xi) \bm{W}_2(\bm{\mu}, \bm{\nu})}.
\]
Combining this estimate with \eqref{eq:shadow-triangle} proves \eqref{eq:smooth-shadow-stability}.

(ii) Starting from \eqref{eq:shadow-triangle}, applying \eqref{eq:smooth-convex-coupling-stability} of \Cref{lem:smooth-coupling-stability} to the previous argument, it follows that
\begin{align*}
    W_2(T_\#\xi,S(\xi; \bm{\nu}))
    &\leq \left(\sum_{i=1}^K \mathbb{E}_{(X_i,Y_i) \sim \pi_i} \|Y_i-T_i(X_i)\|^2\right)^{\frac{1}{2}}\\
    &\leq \left(\sum_{i=1}^K \left( \beta_i W_2(\rho_i, \xi_i) + \sqrt{\frac{\beta_i}{\alpha_i}}W_2(\mu_i, \nu_i) \right)^2 \right)^{\frac{1}{2}}\\
    &\leq \left( \sum_{i=1}^K \beta^2_i W^2_2(\rho_i, \xi_i) \right)^{\frac{1}{2}} + \left(\sum_{i=1}^K \frac{\beta_i}{\alpha_i} W^2_2(\mu_i, \nu_i) \right)^{\frac{1}{2}}\\
    &\leq \max_i \beta_i W_2(\rho, \xi) + \max_{i}\sqrt{\frac{\beta_i}{\alpha_i}} \bm{W}_2(\bm{\mu, \bm{\nu}}),
\end{align*}
which proves \eqref{eq:convex-smooth-shadow-stability}.
\end{proof}

\subsection{Proofs of Corollaries~\ref{cor:shadow-sample-complexity} and~\ref{cor:hp-shadow-complexity}}
Here we present the proofs of \Cref{cor:shadow-sample-complexity} and \Cref{cor:hp-shadow-complexity}.
\begin{proof}[Proof of \Cref{cor:shadow-sample-complexity}]
(i) Recalling $s(\rho_i):=  \sqrt{\mathrm{tr}\,\mathrm{Cov}_{\rho_i}(X_i)}$, observe that 
\[
    \mathbb E s^2(\rho^m_i) = \left(1-\frac1m\right)s^2(\rho_i) \leq s^2(\rho_i),
\]
which implies that $\mathbb E s(\rho^m_i) \leq s(\rho_i)$ by Jensen's inequality. \eqref{eq: convergence rate of FG15} of \Cref{lem: rate of convergence of empirical measures} provides constants independent of the sample sizes such that
\[
    \mathbb E W_2(\rho, \rho^m)\leq C_\rho r_d(m), \quad \mathbb E W_2(\mu_i, \mu_i^{n_i}) \leq C_{\mu_i} r_{d_i}(n_i).
\]
Taking expectation and using the above inequalities prove \eqref{eq:shadow-mean-explicit}.

(ii) An argument similar to the above provides the proof.
\end{proof}

\begin{proof}[Proof of \Cref{cor:hp-shadow-complexity}]
By the product Euclidean norm,
\[
    \sum_{i=1}^K s(\rho_i^m)^2=s(\rho^m)^2, \quad \sum_{i=1}^K\beta_i s(\rho_i^m)W_2(\mu_i,\mu_i^{n_i})\leq\max_i \beta_is(\rho^m) \bm{W}_2(\bm{\mu}, \bm{\mu}^{\bm{n}}).
\]
Thus subadditivity of the square root, \Cref{lem:hp-centered-radius} and \eqref{eq:marginal-wasserstein-bound} yield the deterministic bound
\begin{equation*}\label{eq:hp-deterministic-reduction}
    W_2(\mathcal{S}(\rho; \bm{\mu}), \mathcal{S}(\rho^m; \bm{\mu}^{\bm{n}}))\leq 2\max_i \beta_iW_2(\rho,\rho^m)+\bm{W}_2(\bm{\mu}, \bm{\mu}^{\bm{n}})+2\sqrt{\max_i \beta_i(s(\rho)+W_2(\rho,\rho^m))\bm{W}_2(\bm{\mu}, \bm{\mu}^{\bm{n}})}.
\end{equation*}

Let $A >0$ be a constant determined later. On the event
\begin{equation*}\label{eq:hp-good-event}
    \left\{ W_2(\rho,\rho^m)\leq\frac{\varepsilon}{A}\right\} \bigcap \left( \bigcap_{i=1}^K \left\{W_2(\mu_i,\mu_i^{n_i})\leq\frac{\varepsilon^2}{\sqrt K A^2}\right\} \right),
\end{equation*}
we have $W_2(\rho,\rho^m)\leq1$ and $\bm{W}_2(\bm{\mu}, \bm{\mu}^{\bm{n}})\leq\varepsilon^2/A^2$. Hence
\[
\begin{aligned}
W_2(\mathcal{S}(\rho; \bm{\mu}), \mathcal{S}(\rho^m; \bm{\mu}^{\bm{n}}))
&\leq\frac{2\max_i \beta_i\varepsilon}{A}
+\frac{\varepsilon^2}{A^2}
+\frac{2\sqrt{\max_i \beta_i(s(\rho)+1)}\varepsilon}{A}\\
&\leq\frac{\varepsilon}{A}
\left(2\max_i \beta_i+1+2\sqrt{\max_i \beta_i(s(\rho)+1)}\right)\\
&=\varepsilon
\end{aligned}
\]
by taking $A:=1+2\max_i \beta_i+2\sqrt{\max_i \beta_i(s(\rho)+1)}$. A union bound, followed by \eqref{eq:hp-empirical-tail}, results in
\begin{align*}
    \mathbb{P}(W_2(\mathcal{S}(\rho; \bm{\mu}), \mathcal{S}(\rho^m; \bm{\mu}^{\bm{n}}))>\varepsilon)
\leq{}&
C \exp\left[-c mh_d(\varepsilon^2/A^2)\right]
+C m^{1-\gamma}(\varepsilon^2/A^2)^{-\gamma}\\
&+\sum_{i=1}^K\Bigl\{
C \exp\left[-c n_ih_{d_i}
  (\varepsilon^4/(KA^4))\right]
+C n_i^{1-\gamma}(\varepsilon^4/(KA^4))^{-\gamma}
\Bigr\}.
\end{align*}
For each fixed $a\in(0,1]$ and $t \in (0,1)$, there is some constant $c_{a, \ell} > 0$ such that
\[
    h_d(at)\geq c_{a,\ell}h_d(t).
\]
For $\ell\neq4$ this is immediate. For $\ell=4$, use
\[
\log(2+1/(at))
\leq\log(1/a)+\log(2+1/t)
\leq
\left(1+\frac{\log(1/a)}{\log3}\right)\log(2+1/t).
\]
Absorbing the fixed factors $A,K$ into $C$ and $c$ proves \eqref{eq:hp-shadow-tail}.

Finally, choose $C_*$ sufficiently large in
\eqref{eq:hp-sample-sizes}. Its logarithmic terms make each
exponential contribution in \eqref{eq:hp-shadow-tail}
at most $\frac{\delta}{2(K+1)}$, including the multiplicative constant.
Its moment terms do the same for each polynomial contribution,
since
\[
    \left(D_\delta\varepsilon^{-\frac{2\gamma}{\gamma-1}}\right)^{1-\gamma} \varepsilon^{-2\gamma}=\left(D_\delta\varepsilon^{-4\frac{\gamma}{\gamma-1}}\right)^{1-\gamma} \varepsilon^{-2\gamma}=\frac{\delta}{K+1}.
\]
There are $K+1$ terms of each type, so their sum is at most $\delta$.
Using a union bound proves \eqref{eq:hp-confidence-guarantee}.
\end{proof}

We end this subsection showing a small lemma which is useful for the stability under smoothness only.

\begin{lemma}\label{lem:hp-centered-radius}
For $\eta,\xi\in\mathcal P_2(\mathbb R^\ell)$,
\[
|s(\eta)-s(\xi)|\leq W_2(\eta,\xi).
\]
\end{lemma}
\begin{proof}
By the definition of $s(\eta)$, we have
\[
s(\eta)=\inf_{z\in\mathbb R^\ell}W_2(\eta,\delta_z).
\]
The triangle inequality implies
\[
s(\xi)\leq W_2(\xi,\eta)+s(\eta).
\]
Interchanging $\eta$ and $\xi$ proves the assertion.
\end{proof}

\subsection{Proof of Theorem~\ref{thm:uniqueness of projection}}
\begin{proof}[Proof of \Cref{thm:uniqueness of projection}]
For each $i$, there is a $(\mathrm{Id},T_i)_\#\rho_i$ the unique optimal transport plan from $\rho_i$ to $\mu_i$ achieving $C_i(\rho_i, \mu_i)$ where $T_i$ is a unique transport map $\rho_i$-almost everywhere sense guaranteed by \Cref{lemma: uniqueness of tranport plan}. Then, $\mathcal{S}_{\bm{c}}(\rho; \bm{\mu}) = T_\#(\rho)$ where 
\[
    T(\bm{x}) = (T_1(x_1), \dots, T_K(x_K)).
\]
In particular, $\mathcal{S}_{\bm{c}}(\rho; \bm{\mu})$ exists uniquely since each $T_i$ is unique. By the construction of $\mathcal{S}_{\bm{c}}(\rho; \bm{\mu})$, it is direct that $\mathcal{S}_{\bm{c}}(\rho; \bm{\mu})$ is a solution for \eqref{problem: optimal transport projection} achieving the minimum $\sum C_i(\rho_i, \mu_i)$.

Let $\tilde{\pi} \in \Pi(\bm{\mu})$ be another optimizer for \eqref{problem: optimal transport projection}. Let $\gamma$ be an optimal coupling between $\rho$ and $\tilde{\pi}$. Let $\gamma_i$ denote the $(\mathcal{X}_i \times \mathcal{X}_i)$-marginal of $\gamma$. By the separability of $\bm{c}$, we have
\[
    \int_{\bm{\mathcal{X}} \times \bm{\mathcal{X}}} \bm{c}(\bm{x}, \bm{y}) d\gamma (\bm{x}, \bm{y}) = \sum_{i=1}^K \int_{\mathcal{X}_i \times \mathcal{X}_i} c_i(x_i, y_i) d\gamma_i(x_i, y_i) \geq \sum_{i=1}^K C_i(\rho_i, \mu_i).
\]
Since $C_i(\rho_i, \mu_i)$ is the minimum value for each $i=1, \dots, K$, the last inequality becomes the equality if $\gamma_i$ is an optimal transport plan for $C_i(\rho_i, \mu_i)$. Since there is a unique optimal transport plan for $C_i(\rho_i, \mu_i)$, $\gamma_i = (\mathrm{Id},T_i)_\#\rho_i$, which implies that $\tilde{\pi} = T_\#(\rho) \mathcal{S}_{\bm{c}}(\rho; \bm{\mu})$. This completes the proof.
\end{proof}

\begin{lemma}\cite[Theorem 3.7]{gangbo1996geometry}\label{lemma: uniqueness of tranport plan}
For $\mu, \nu \in \mathcal{P}(\mathbb{R}^d)$, assume $c=h$ satisfies Assumption~\ref{assumption: strict convexity of cost function} and \eqref{eq: moment condition for h}. If $\mu$ is absolutely continuous, then there is a unique optimal transport plan $\pi^* = (\mathrm{Id}, T)_\# (\mu)$ such that
\[
    T(\bm{x}) = x - \nabla h^*(\nabla \varphi) (x)
\]
where $h^*$ is a convex conjugate of $h$, and $\varphi$ is an optimal dual potential from $\mu$ to $\nu$.
\end{lemma}

\section{Proof for transfer learning}
\subsection{Proofs of Propositions~\ref{prop:approximation-shadow-decomposition} and~\ref{prop: qualitative consistency}}
We start with the following preparatory lemma.

\begin{lemma}\label{lem:root-risk-comparison}
Let $\lambda,\widetilde \lambda \in \mathcal P_2(\mathbb R^{d_1 + d_2})$ and $g\in\mathcal G_L.$ Then
\begin{equation}
\label{eq:root-risk-comparison}
\begin{aligned}
\Bigg|
\left(
\int\|v-g(u)\|^2\,d\lambda(u,v)
\right)^{1/2}
-
\left(
\int\|v-g(u)\|^2\,d\widetilde \lambda(u,v)
\right)^{1/2}
\Bigg| \leq
\sqrt{1+L^2}\,W_2(\lambda,\widetilde \lambda).
\end{aligned}
\end{equation}
\end{lemma}

\begin{proof}
Take any coupling of $(U,V)\sim \lambda$ and $(\widetilde U,\widetilde V)\sim\widetilde \lambda.$ The triangle inequality, Lipschitz continuity of $g$ and Hölder's inequality give
\begin{align*}
\|(V-g(U))-(\widetilde V-g(\widetilde U))\| \leq
\|V-\widetilde V\|+L\|U-\widetilde U\| \leq
\sqrt{1+L^2}
\left(
\|V-\widetilde V\|^2+\|U-\widetilde U\|^2
\right)^{1/2}.
\end{align*}
Fix $\pi \in \Pi(\lambda, \widetilde \lambda)$, applying the reverse triangle inequality in $L^2(\pi)$ yields that
\begin{align*}
    \left| \|(v-g(u))\|_{L^2(\lambda)} - \|(\widetilde v-g(\widetilde u))\|_{L^2(\widetilde \lambda)} \right| &\leq \|(v-g(u))-(\widetilde v-g(\widetilde u))\|_{L^2(\pi)}\\
    &\leq \sqrt{1+L^2} \left( \|v-\widetilde v\|^2_{L^2(\pi)}+\|u-\widetilde u\|^2_{L^2(\pi)} \right)^{\frac{1}{2}} .
\end{align*}
Minimizing the resulting bound with respect to $\pi$ proves
\eqref{eq:root-risk-comparison}.
\end{proof}

\begin{proof}[Proof of \Cref{prop:approximation-shadow-decomposition}]
First, by \Cref{thm: Brenier's thm}, there are optimal transport maps from $\rho_1$ to $\mu_1$ and from $\rho_2$ to $\mu_2$. Furthermore, since $\mu_1$ is also absolutely continuous, this optimal transport map from $\rho_1$ to $\mu_1$ has its inverse $\mu_1$-almost surely, which implies that $G$ is well defined in \eqref{eq: construction of G}.

By \Cref{lem: existence of regular solution}, for any realization of samples, there is a solution for \eqref{eq: minimization over strongly convex/smooth functions} with given $\alpha=0$ and $0 <\beta < \infty$. Hence $\hat{T}_i$'s are well defined too.

Observe that for any $g : \mathcal{X}_1 \rightarrow \mathcal{X}_2$,
\[
 \int\|v- g(u)\|^2 \,dS(\rho^m;\widehat{\bm\nu})(u,v) = \frac{1}{m}\sum_{j=1}^m \left\| \widehat T_2(Y_j)-g(\widehat T_1(X_j)) \right\|^2.
\]
Fix any $g\in\mathcal G_L$. Applying \eqref{eq:root-risk-comparison} of \Cref{lem:root-risk-comparison} to $\widehat G_L$ with
\[
    \lambda = \mathcal{S}(\rho;\bm{\mu}), \quad \widetilde \lambda = \mathcal{S}(\rho^m; \widehat{\bm{\nu}}),
\]
it holds that
\begin{align*}
    \|\widehat G_L-G\|_{L^2(\mu_1)} &\leq \left( \int\|v-\widehat G_L(u)\|^2 \,dS(\rho^m;\widehat{\bm\nu})(u,v) \right)^{1/2} + \sqrt{1+L^2}\, W_2 (\mathcal{S}(\rho;\bm{\mu}), \mathcal{S}(\rho^m; \widehat{\bm{\nu}}) )\\
    &\leq \left( \int\|v- g(u)\|^2 \,dS(\rho^m;\widehat{\bm\nu})(u,v) \right)^{1/2} + \sqrt{1+L^2}\, W_2 (\mathcal{S}(\rho;\bm{\mu}), \mathcal{S}(\rho^m; \widehat{\bm{\nu}}) ).
\end{align*}
Applying \eqref{eq:root-risk-comparison} of \Cref{lem:root-risk-comparison} again to the first term in the right-hand side above, and taking the infimum over $g\in\mathcal G_L$, the conclusion follows.
\end{proof}

A next lemma states the denseness of Lipschitz functions in $L^2$.

\begin{lemma}\label{lem:lipschitz-density}
For any fixed $G\in L^2(\mu_1;\mathbb R^{d_2})$,
\[
A_L(G)\longrightarrow0
\qquad\text{as }L\to\infty.
\]
\end{lemma}

\begin{proof}
Bounded globally Lipschitz functions are dense in $L^2(\mu_1)$.
For a more general density statement, see \citet[Corollary~3.4]{Pasqualetto2025}. In Euclidean space, this applies to every finite Borel measure.

Applying the density statement componentwise gives globally
Lipschitz functions $g_j:\mathbb R^{d_1}\to\mathbb R^{d_2}$ such that
\[
\|G-g_j\|_{L^2(\mu_1)}\longrightarrow0.
\]
For every fixed $j$ and every
$L\geq\operatorname{Lip}(g_j)$,
\[
A_L(G)\leq\|G-g_j\|_{L^2(\mu_1)}.
\]
First take the upper limit as $L\to\infty$, and then let
$j\to\infty$.
\end{proof}

\begin{proof}[Proof of \Cref{prop: qualitative consistency}]
Thanks to the condition that $L  W_2 (\mathcal{S}(\rho;\bm{\mu}), \mathcal{S}(\rho^m; \widehat{\bm{\nu}}) ) \rightarrow 0$ and \Cref{lem:lipschitz-density}, it suffices to show the convergence of $\mathcal{S}(\rho^m; \widehat{\bm{\nu}})$ to $\mathcal{S}(\rho;\bm{\mu})$.

Using \eqref{eq:smooth-shadow-stability} of \Cref{thm:smooth-shadow-stability}, we have
\[
    W_2(\mathcal{S}(\rho;\bm{\mu}), \mathcal{S}(\rho^m; \widehat{\bm{\nu}})) \leq 2 \beta W_2(\rho, \rho^m) + \bm{W}_2(\bm{\mu}, \widehat{\bm{\nu}}) +  2 \sqrt{\beta s(\rho^m) \bm{W}_2(\bm{\mu}, \widehat{\bm{\nu}})}
\]
Recalling the fact that $s(\rho^m) \rightarrow s(\rho)$ almost surely and
\[
    \bm{W}_2(\bm{\mu}, \widehat{\bm{\nu}}) = \left( \sum_{i=1}^2 W_2^2(\mu_i, (\nabla \hat{\varphi}_i)_\#(\rho_i^m)) \right)^{\frac{1}{2}},
\]
the problem reduces to showing the convergence of $W_2^2(\mu_i, (\nabla \hat{\varphi}_i)_\#(\rho_i^m))$ for $i=1,2$.

By the optimality of $\hat{\varphi}_i$ for \eqref{eq: minimization over strongly convex/smooth functions}, the assumption of $\Phi^i_{0, \beta} $ and the triangle inequality,
\begin{align*}
    W_2(\mu_i, (\nabla \hat{\varphi}_i)_\#(\rho_i^m)) &\leq W_2(\mu_i, \mu_i^{n_i}) + W_2(\mu_i^{n_i}, (\nabla \hat{\varphi}_i)_\#(\rho_i^m))\\
    &\leq W_2(\mu_i, \mu_i^{n_i}) + W_2(\mu_i^{n_i}, (\nabla \varphi_i)_\#(\rho_i^m)).
\end{align*}
Notice that $\mu_i^{n_i}$ and $(\nabla \varphi_i)_\#(\rho_i^m)$ are independent. Furthermore since $(\nabla \varphi_i)_\#(\rho_i) = \mu_i$, one has
\[
    (\nabla \varphi_i)_\#(\rho_i^m)) \overset{d}{=} \tilde{\mu}_i^m
\]
where $\tilde{\mu}_i^m$ is an empirical distribution of $m$ samples drawn by $\mu_i$ independently of the target samples. Since $W_2 \geq 0$, 
\[
    \left(\sum_{i=1}^2 W_2^2(\mu_i, (\nabla \hat{\varphi}_i)_\#(\rho_i^m)) \right)^{\frac{1}{2}}  \leq \sum_{i=1}^2 W_2(\mu_i, (\nabla \hat{\varphi}_i)_\#(\rho_i^m)) .
\]
The triangle inequality yields
\begin{equation}\label{eq: upper bound of mu_i and transported rho_i}
    \left(\sum_{i=1}^2 W_2^2(\mu_i, (\nabla \hat{\varphi}_i)_\#(\rho_i^m)) \right)^{\frac{1}{2}} \leq \sum_{i=1}^2 2 W_2(\mu_i, \mu_i^{n_i}) + W_2(\mu_i, \tilde{\mu}_i^m),
\end{equation}
which shows $\bm{W}_2(\bm{\mu}, \widehat{\bm{\nu}}) \rightarrow 0$ as $n_1, n_2, m \rightarrow 0$. This completes the proof.
\end{proof}

\subsection{Proof of Theorem~\ref{thm: statistics of transfer learning}}
\begin{proof}[Proof of \Cref{thm: statistics of transfer learning}]
(i) Observe that, using the fact that $\mu_i = (\nabla \varphi_i)_\# (\rho_i)$ and $\beta_i$-smoothness of $\varphi_i$,
\[
    W_2(\mu_i, \tilde{\mu}_i^m) = W_2((\nabla \varphi_i)_\# (\rho_i), (\nabla \varphi_i)_\# (\rho_i^m)) \leq \beta_i W_2(\rho_i, \rho_i^m).
\]
Recalling \eqref{eq: upper bound of mu_i and transported rho_i}, leveraging the above implies
\[
    \left(\sum_{i=1}^2 W_2^2(\mu_i, (\nabla \hat{\varphi}_i)_\#(\rho_i^m)) \right)^{\frac{1}{2}} \leq \sum_{i=1}^2 2 W_2(\mu_i, \mu_i^{n_i}) +\beta_i W_2(\rho_i, \rho_i^m).
\]
Also note that
since $\rho_2 = F_\#(\rho_1)$ and $F$ is $L$-Lipschitz,
\[
    W_2(\rho, \rho^m) \leq \sqrt{1 +L^2} W_2(\rho_1, \rho_1^m), \quad W_2(\rho_2, \rho_2^m) \leq L W_2(\rho_1, \rho_1^m).
\]
Applying two observations for \eqref{eq:convex-shadow-mean-explicit} results in
\begin{equation}\label{eq: upper bound of G - 
hat{G}}
    \|G - \widehat{G}\|_{L^2(\mu_1)} \leq \sqrt{1+\left( \frac{\beta}{\alpha}L \right)^2} \left\{ \left( 2 \beta \sqrt{1 + L^2} + \sqrt{\frac{\beta}{\alpha}} (\beta + \sqrt{1 + L^2}) \right) W_2(\rho_1, \rho_1^m)  + \sqrt{\frac{\beta}{\alpha}}  \sum_{i=1}^2 2 W_2(\mu_i, \mu_i^{n_i}) \right\}.
\end{equation}
Leveraging \eqref{eq: convergence rate of FG15} of \Cref{lem: rate of convergence of empirical measures}, the conclusion follows.

(ii) For $0<t\leq1$, suppose simultaneously that
\[
W_2(\rho_1,\rho_1^m)\leq t,
\quad
W_2(\mu_i,\mu_i^{n_i})\leq t 
\text{ for $i=1,2$}.
\]
Substituting directly into \eqref{eq: upper bound of G - hat{G}} gives
\begin{align*}
\frac{\|G-\widehat G\|_{L^2(\mu_1)}}
     {\sqrt{1+\left( \frac{\beta}{\alpha}L \right)^2}} \leq
\left( 2 \beta \sqrt{1 + L^2} + \sqrt{\frac{\beta}{\alpha}} (\beta + \sqrt{1 + L^2}) \right) t + \sqrt{\frac{\beta}{\alpha}}  \sum_{i=1}^2 2t   =: Ct.
\end{align*}
Then, the union bound yields
\begin{align*}
\mathbb P\!\left(
\frac{\|G-\widehat G\|_{L^2(\mu_1)}}
     {\sqrt{1+\left( \frac{\beta}{\alpha}L \right)^2}}
>Ct\right)  \leq \mathbb P(W_2(\rho_1,\rho_1^m)>t) +
\sum_{i=1}^2 \mathbb P(W_2(\mu_i,\mu_i^{n_i})>t) .
\end{align*}

Applying \Cref{lem: rate of convergence of empirical measures} (ii) with threshold $t^2$ provides
\begin{align*}
\mathbb P(W_2(\rho_1,\rho_1^m)>t)
&\leq  Ce^{-cmh_{d_1}(t^2)}\mathbf 1_{\{t\leq1\}} +Cm^{1-\gamma}t^{-2\gamma},\\
\mathbb P(W_2(\mu_i,\mu_i^{n_i})>t)
&\leq Ce^{-cn_ih_{d_i}(t^2)}\mathbf 1_{\{t\leq1\}} +Cn_i^{1-\gamma}t^{-2\gamma}.
\end{align*}
Substitution in the union bound proves the result.

(iii) The exponential requirements in
\eqref{eq:transfer-n-m-choice} imply
\begin{align*}
Ce^{-cmh_{d_1}(t^2)}\leq\frac{\delta}{3},\quad
Ce^{-cn_i h_{d_i}(t^2)}\leq\frac{\delta}{3}
\text{ for $i=1,2$}.
\end{align*}
Indeed, each exponent is at least $\log(3C/\delta)$.

Since $\gamma>2$, the polynomial requirement on $m$ gives
\[
m^{\gamma-1}
\geq\frac{3C}{\delta}\,t^{-2\gamma},
\quad
Cm^{1-\gamma}t^{-2\gamma}
\leq\frac{\delta}{3}.
\]
Since $t\leq1$, it also follows that
\[
    Cm^{1-\gamma}t^{-2\gamma}
\leq Cm^{1-\gamma}t^{-2\gamma}
\leq\frac{\delta}{3}, \quad Cn_i^{1-\gamma}t^{-2\gamma}
\leq\frac{\delta}{3} \text{ for $i=1,2$}.
\]
These choices imply the conclusion
\end{proof}

\section{Conclusion}
In this paper, we show the quantitative stability of the shadow projection over Wasserstein space. Under the assumption on the regularity of Brenier potentials, which is satisfied by a family of strongly log-concave measures, \(\frac{1}{2}\)-Hölder continuity of the shadow holds with respect to input data. As an application of it, we apply the shadow framework to transfer learning with domain adaptation with unpaired target covariate and response observations. Leveraging the shadow framework, we obtain the sample complexity, high probability tail bound, and the required sample sizes for transfer learning.

We propose two future directions. The first is the extension of the stability of shadow beyond the regularity assumption. The present stability bounds strongly relies on the regularity of Brenier potential. For general shadow, how does one achieve its stability? Note that without absolute continuity, multiple shadows exist, which is already a barrier for stability. Furthermore, a shadow is generally described by a transport kernel in this case. Hence, the problem reduces to achieving the stability of transport plan.

The second direction is to develop a more realistic transfer learning framework. Our model pays attention to the noiseless setting, which is not appropriate under some context. For example, measurement error and intrinsic variability introduce noise in most biological data. In order to handle transfer learning in these settings, it is necessary to cover the case with noise. In noisy setting, our estimator $\hat{G}$ might overfit the data. So the question is which model is natural for this noisy transfer learning setting, and how do we understand it?

\section*{Acknowledgements}
JK is supported by CUHK-SZ start-up UDF03004229. The author is grateful to Wooseok Ha for introducing a domain adaptation problem in transfer learning when the author visited KAIST.

\section*{AI use}
Most of ideas in this paper are human-generated and also developed with assistance from ChatGPT (OpenAI), based on the human author's calculation and suggestion. The human author checks and reorganizes AI-generated proofs.

\printbibliography

\begin{appendices}
\section{Basic properties of log-concave measures}
\begin{proposition}\label{prop: some properties of log-concave}
Let $d=d_1+d_2$, and let $\rho$ be a probability measure on
$\mathbb{R}^{d_1}\times\mathbb{R}^{d_2}$ with density
\[
\rho(dx\,dy)=\frac{1}{Z}e^{-W(x,y)}\,dx\,dy.
\]
Suppose that, for some $0 < \alpha < \beta$, $\alpha \mathrm{I}_{d} \preceq D^2W \preceq \beta \mathrm{I}_{d} $ in the distributional sense. Let $\rho_1$ be the marginal of $\rho$ on $\mathbb{R}^{d_1}$, with potential
\[
W_1(x)
:=-\log\int_{\mathbb{R}^{d_2}}e^{-W(x,y)}\,dy.
\]
\begin{enumerate}
    \item[(i)] $D^2W_1\succeq \alpha \mathrm{I}_{d_1}$ in the distributional sense. In particular, the marginal $\rho_1$ is $\alpha$-strongly log-concave.
    \item[(ii)] $D^2W_1\preceq \beta \mathrm{I}_{d_1}$ in the distributional sense.
\end{enumerate}

\begin{proof}
(i) 
The assumption implies that $(x,y)\longmapsto
W(x,y)-\frac{\alpha}{2}
(\|x\|^2+\|y\|^2)$ is convex. Consequently, the function
\[
f(x,y)
:=
\exp\left(
-W(x,y)+\frac{\alpha}{2}\|x\|^2
\right)
\]
is log-concave, since $\log f(x,y)$ is concave.

By Pr\'ekopa's theorem, the function
\[
x\longmapsto
\int_{\mathbb{R}^{d_2}}f(x,y)\,dy
=
\exp\left(
-W_1(x)+\frac{\alpha}{2}\|x\|^2
\right)
\]
is log-concave. Therefore, $x\longmapsto
W_1(x)-\frac{\alpha}{2}\|x\|^2$ is convex, which yields $D^2W_1\succeq\alpha \mathrm{I}_{d_1}.$

(ii) The assumption implies that, for every fixed
$y\in\mathbb{R}^{d_2}$, the function $x\longmapsto W(x,y)-\frac{\beta}{2}\|x\|^2$ is concave. Hence
\[
f(x,y)
:=
\exp\left(
-W(x,y)+\frac{\beta}{2}\|x\|^2
\right)
\]
is log-convex in $x$.

Define
\[
F(x)
:=
\int_{\mathbb{R}^{d_2}}f(x,y)\,dy.
\]
For any $x_0,x_1\in\mathbb{R}^{d_1}$ and
$t\in[0,1]$, log-convexity gives
\[
f((1-t)x_0+tx_1,y)
\leq
f(x_0,y)^{1-t}f(x_1,y)^t.
\]
Integrating and applying H\"older's inequality,
we obtain
\begin{align*}
    F((1-t)x_0+tx_1) \leq \int_{\mathbb{R}^{d_2}} f(x_0,y)^{1-t}f(x_1,y)^t\,dy \leq F(x_0)^{1-t}F(x_1)^t.
\end{align*}
Thus $F$ is log-convex. Since
\[
    \log F(x) = -W_1(x)+\frac{\beta}{2}\|x\|^2,
\]
the function $x\longmapsto W_1(x)-\frac{\beta}{2}\|x\|^2$ is concave. Therefore, $D^2W_1\preceq\beta \mathrm{I}_{d_1}.$
\end{proof}
\end{proposition}

\end{appendices}

\end{document}